\documentclass[12pt,a4paper,reqno]{amsart}

\usepackage{amsmath,amsfonts,amssymb,amsthm,color,bm}
\usepackage{esint}
\usepackage{dsfont}
\usepackage[utf8]{inputenc}
\usepackage{graphicx}
\usepackage{hyperref}
\usepackage{color}
\usepackage{xcolor}
\usepackage{mathrsfs}
\hypersetup{
    colorlinks=true,
    linkcolor=blue,
    citecolor=blue,
    filecolor=magenta,      
    urlcolor=cyan,
}

\usepackage[font=tiny,labelfont=bf]{caption}

\theoremstyle{plain}
\newtheorem{theorem}{Theorem}[section]

\newtheorem{lemma}[theorem]{Lemma}
\newtheorem{proposition}[theorem]{Proposition}

\theoremstyle{definition}

\newtheorem*{theorem*}{Theorem}

\numberwithin{equation}{section}

\newcommand{\R}{\mathbb{R}}
\renewcommand{\S}{\mathbf{S}}
\newcommand{\I}{\mathbf{I}}
\newcommand{\J}{\mathbf{J}}
\renewcommand{\L}{\hbox{\raisebox{0.06em}-}\kern-0.45emL}
\renewcommand{\P}{\mathbf{P}}
\newcommand{\N}{\mathbb{N}}
\newcommand{\T}{\mathbf{T}}

\newcommand{\br}[1]{\left({#1}\right)}
\newcommand{\Br}[1]{\left[{#1}\right]}
\newcommand{\BR}[1]{\left\{{#1}\right\}}
\newcommand{\nrm}[1]{\left\|{#1}\right\|}
\newcommand{\abs}[1]{\left|{#1}\right|}
\newcommand{\ang}[1]{\left<{#1}\right>}

\newcommand{\ve}{{\bm{e}}}

\newcommand{\valpha}{{\bm{\alpha}}}
\newcommand{\vbeta}{{\bm{\beta}}}
\newcommand{\vgamma}{{\bm{\gamma}}}

\newcommand{\vxi}{{\bm{\xi}}}

\DeclareRobustCommand{\rchi}{{\mathpalette\irchi\relax}}
\newcommand{\irchi}[2]{\raisebox{\depth}{$#1\chi$}} % inner command, used by \rchi

\newcommand{\supp}{\operatorname{supp}}

\newcommand{\Dil}{D}
\newcommand{\Mod}{M}
\newcommand{\Tr}{T}
\newcommand{\Proj}{P}

\title[bilinear H\"ormander multipliers with Lipschitz singularities]{A sharp  H\"ormander condition for bilinear Fourier multipliers with Lipschitz singularities}

\author{Jiao Chen, Martin Hsu, Fred Yu-Hsiang Lin}

\begin{document}

\begin{abstract}
    This paper studies the $L^{p}$ boundedness of bilinear Fourier  multipliers 
 in the local $L^{2}$ range.
 We assume
 a H\"ormander condition relative to a singular set that is a finite union of Lipschitz curves. The H\"ormander condition is  sharp with respect to the Sobolev exponent. Our setup  generalizes the non-degenerate bilinear Hilbert transform but avoids issues 
 of uniform bounds 
 near degeneracy.
 
\end{abstract}

\maketitle

\normalsize
\tableofcontents

\section{Introduction}

An $n$-linear Fourier multiplier $m$  is a function on
the space $V$ of all points $\xi =(\xi_{1},\cdots ,\xi_{n+1})\in \R^{n+1}$ such that
\[\sum_{j=1}^{n+1}\xi_j=0.\]
It is associated with an $(n+1)$-linear form acting on functions on the real line defined by
\begin{equation}\label{eq_trilinear_form}
        \Lambda_{m} (f_{1},\cdots,f_{n+1}):=\int_{V}m(\xi)\prod_{j=1}^{n+1}\widehat{f}_{j}(\xi_{j})d\mathcal{H}^{n}(\xi).
    \end{equation}
Here $\mathcal{H}^{n}$
denotes the $n$-dimension Hausdorff measure on $V$.

We call such a multiplier $n$-linear as classically one associates to it an $n$-linear operator dual to this $(n+1)$-linear form.

For a multiplier $m$ and a tuple $p=(p_{1},\cdots,p_{n+1})$ of Lebesgue norm exponents in $(1,\infty)^{n+1}$  with
\begin{equation}\label{holderexp}
    \sum_{j=1}^{n+1} \frac 1{p_j}=1,
\end{equation}
we define the constant
$\mathcal{C}(m,p)$ to be the infimum of all constants $C>0$ satisfying 
 \begin{equation}\label{goalbdd}
        |\Lambda_{m}(f_{1},\cdots, f_{n+1})|\leq C \prod_{j=1}^{n+1}\|f_{j}\|_{L^{p_{j}}}
    \end{equation}
 for all tuples of Schwartz functions $f_j$.

We say the form $\Lambda_{m}$ is bounded in the open Banach range if $\mathcal{C}(m,p)$ is finite on all tuples $p$ with \eqref{holderexp} in the range $(1,\infty)^{n+1}$.
We say it is bounded in the 
local $L^{2}$ range if it is bounded
for all tuples $p$ with \eqref{holderexp} in $(2,\infty)^{n+1}$.

Classical works concern classes of
multipliers singular at one point, typically the origin.
These include the Mikhlin class $M_s$, which is all 
multipliers satisfying away from the origin the symbol bounds
\begin{equation}\label{mihlincond}
    |(\partial^{\alpha}m)(\xi)|\leq C_{\alpha}|\xi|^{-\alpha}
\end{equation}
 for all multi-indices $\alpha$ up to order $|\alpha|\le s$.
 Another slightly larger class is H\"{o}rmander class $H_s$, which is the set of all multipliers satisfying 
\begin{equation}\label{Hormcondclassic}
    \operatorname{sup}_{j\in \mathbb{Z}}\|m(2^{j}\cdot)\Psi\|_{H^{s}(V)}<\infty
\end{equation}
where $\Psi$ is a smooth bump function compactly supported away from $0$. Here $H^{s}$ is the $L^{2}$-based inhomogeneous Sobolev norm defined by
\begin{equation}
    \|f\|_{H^{s}(V)}:=\left\|(1+|x|^{2})^{\frac{s}{2}}\widehat{f\cdot d\mathcal{H}^{n}}(x)   \right\|_{L^{2}_{x}(V)},
\end{equation}
where
\begin{equation}
    \widehat{f\cdot d\mathcal{H}^{n}}(x):=\int_{V}f(y)e^{-2\pi ix\cdot y}d\mathcal{H}^{n}(y).
\end{equation}
 
The classical Mikhlin multiplier theorem \cite{mihlin1956} gives boundedness in the open Banach range for linear multipliers $m\in M_{1}$.
 %\[|\alpha|\le \lfloor\frac{d}{2} \rfloor+1\],.
  In \cite{hormander1960estimates}, H\"{o}rmander proved  
  boundedness in the open Banach range for linear multipliers $m$ in $H_s$ with $s>\frac 12$.
  Boundedness for the general
 $n$-linear case in the open Banach range was shown by Coifman and Meyer  \cite{coifman1978commutateurs}, \cite{coifman1978dela} 
  for  $m\in M_{s}$ with $s$ sufficiently large  and by  Tomita \cite{tomita2010hormander} for $m\in H_s$ with the sharp
  condition $s>\frac n2$.
For results concerning exponents outside the Banach range, see \cite{kenig1999multilinear},\cite{grafakos2002multilinear},\cite{grafakos2012hormander},\cite{lee2021hormander}.

More recently, people studied multilinear multipliers with higher dimensional singularities. 
Lacey and Thiele \cite{lt1997} proved bounds in the local $L^2$ range for $n=2$ and $m=\rm{sgn}(\alpha_1\xi_1+\alpha_2\xi_2)$,
the so-called bilinear Hilbert transform, for all vectors $\alpha=(\alpha_1,\alpha_2)$. The bound is non-trivial only for $\alpha$ outside the three so-called degenerate one-dimensional subspaces.
This result was extended to the open Banach range and beyond in \cite{lacey1999calderon}. That this $m$ is a particular instance
of more general multipliers singular along a line was noted by Gilbert and Nahmod \cite{gilbert2000boundedness}, who extended the result accordingly.
Muscalu, Tao, and Thiele \cite{muscalu2002multi} proved bounds in the open Banach  range for $n$-linear multipliers satisfying
\begin{equation}\label{mttcondition}
|\partial^{\alpha}m(\xi)|\lesssim \operatorname{dist}(\xi,\Gamma)^{-|\alpha|}
\end{equation}
for singularity $\Gamma$ a non-degenerate subspace with $\operatorname{dim}{\Gamma}<\frac{n+1}{2}$ and for $\alpha$ up to some large degree that has not been specified in \cite{muscalu2002multi}.

The bounds in \cite{lt1997} are not uniform in $\alpha$. Uniform bounds were proven in \cite{grafakos2004uniform}, \cite{li2006uniform} by Grafakos and Li, and later the range was extended by Uraltsev and Warchalski in \cite{uraltsev2022full}.

Curved singularites were first studied by Muscalu \cite{muscalu2000p}. 
Later, the bound of bilinear disk multiplier was obtained by Grafakos and Li \cite{grafakos2006disc}.

The main theorem of this paper establishes the sharp Sobolev exponent for the H\"ormander condition associated with bilinear
multipliers whose singularities are unions of Lipschitz curves away from degenerate directions.
This is the first work that provides this sharp H\"ormander condition for multilinear multipliers
with singularities of dimension larger than zero.
Moreover, we work in a continuous model without discretization in the vein of \cite{do2015lp} and develop a suitable setting to analyze the geometry arising from the presence of Lipschitz singularity.

Define dilation, translation, and modulation operators
\begin{equation*}
    (D_{a}^{p}f)(x):=a^{-\frac{n}{p}}f\left(\frac{x}{a}\right)
\end{equation*}
\begin{equation*}
    (T_{a}f)(x):=f(x-a)
\end{equation*}
\begin{equation*}
    (M_{a}f)(x):=e^{2\pi iax}f(x).
\end{equation*}
For $\vbeta \in V$, define the distance function $d_\Gamma\br{\vbeta}:=\inf_{\vxi\in\Gamma}\abs{\vbeta-\vxi}$. Let $B_{r}(x)$ denote the open ball with radius $r$ centered at $x$. Let $\eta$ be a $L^{1}$ normalized function supported on $[-1,1]$ defined by
 \[\eta(x)=\br{\int_{-1}^1 e^{\frac {-1}{1-t^2}} dt}^{-1}\cdot e^{\frac {-1}{1-x^2}}\cdot 1_{\Br{-1,1}}(x).\]
 Define
  \begin{equation*}
      \widetilde{\eta}:=1_{B_{\frac{3}{20}}(0)} \ast {\Dil}^{1}_{\frac{1}{100}}\eta
  \end{equation*}
  which is constant one in $B_{\frac{1}{10}}(0)$ and supported on $B_{\frac{2}{10}}(0)$.
Define a smooth function $\Phi$ on $V$ 
\begin{equation*}
    \Phi (x):=\widetilde{\eta}(|x|).
\end{equation*}
For a subspace $A\subseteq \mathbb{R}^{n}$ and a vector $v\in \mathbb{R}^{n}$, we denote the orthogonal projection of $v$ onto $A$ as $P_{A}v$. Let $0\leq \theta_0<\frac \pi 6$.
For $j\in \{1,2,3\}$, let $\mathcal{K}_j(\theta_0)$ be the open double cone of all vectors $\vbeta$ in $V$ which have angle less than
   $\theta_0$
   to the line spanned by by $P_V e_j$, i.e., as the length of $P_Ve_j$ is $\frac {\sqrt{6}}{3}$, $\mathcal{K}_j(\theta_0)$ contains points $\vbeta \in V$ satisfying
\begin{equation}\label{stay_in_one_cond}
        |\ang{\vbeta, e_j}|= |\ang{\vbeta, P_V e_j}|>\frac{\sqrt{6}}{3}\abs{\vbeta}\operatorname{cos}\theta_{0} .
    \end{equation} 

\begin{theorem}\label{mainthm}
   Let $n=2$.  Let $2<p_{1},p_{2},p_{3}<\infty$ with  $\frac{1}{p_{1}}+\frac{1}{p_{2}}+\frac{1}{p_{3}}=1$.
   Let  $0\leq \theta_0<\frac \pi 6.$ Let  $s>1$. There is a constant $C(p_1,p_2,p_3,\theta_{0},s,N)$ such that the following holds.
   
   For every $1\le \iota\le N$, let $\Gamma_\iota\subset V$ be a closed set such that there exists an index $j\in \{1,2,3\}$ such that
   for every distinct $\vgamma, \vgamma'\in \Gamma_\iota $, we have $\vgamma-\vgamma'\in \mathcal{K}_j(\theta_0)$.
  Let $\Gamma$ be the union of the sets $\Gamma_\iota$ for $1\le \iota \le N$. Let \(m\) be a function on \(V\) satisfying 
    \begin{equation}\label{Hormcond}
        \sup_{\vbeta \in V\setminus \Gamma}\left\|\left( {\Dil}^{\infty}_{d_{\Gamma}(\vbeta)^{-1}}{\Tr}_{-\vbeta}m \right) \cdot \Phi\right\|_{H^{s}(V)}\leq 1.
    \end{equation}
    Then we have for the form bound \eqref{goalbdd} the inequality
    \[\mathcal{C}(m,p_1,p_2,p_3)\le 
C(p_1,p_2,p_3,\theta_{0},s,N).\]

\end{theorem}
Note that the theorem applies in particular to the case \[m(\xi_{1},\xi_{2},\xi_{3})=\widetilde{m}(\xi_{1}-\xi_{2})\] 
with $\widetilde{m}$ satisfying the H\"{o}rmander condition on real line for $s>1$.

\begin{figure}
    \centering
    \includegraphics[width=0.4\textwidth,height=6cm]{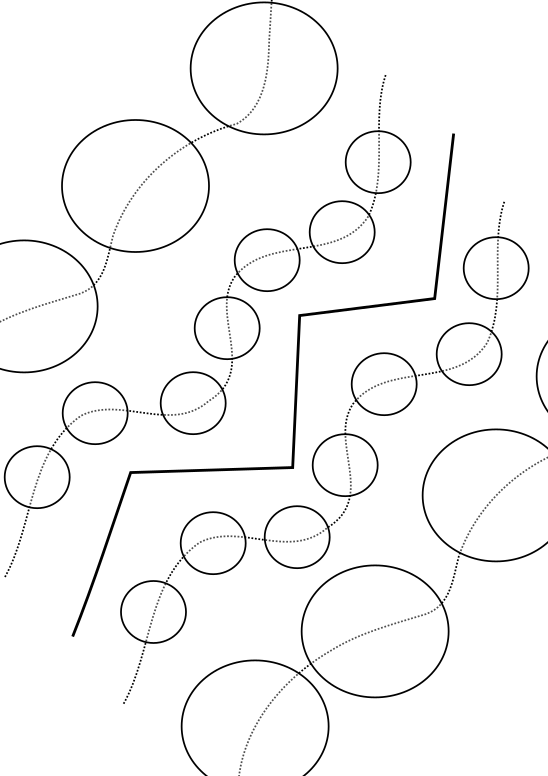}
    \caption{We may view \eqref{Hormcond} as testing Sobolev norm of $m$ on scaled Whitney bumps.}
    \label{fig:enter-label}
\end{figure}

The exponent $s$ in this theorem is sharp. It 
suffices to establish the sharpness of the exponent $s$ in the case where $\Gamma$ is a point. This sharpness has been
discussed in \cite{grafakos2018hormander}.

The assumption \eqref{stay_in_one_cond} says that the tangencies of each Lipschitz curve stay away from a fixed angle from the degenerate direction. 
While bounds for specific examples such as circular arcs with degenerate tangencies have been established in the literature
\cite{grafakos2006disc}, even the question for convex arcs in general with degenerate tangencies appears to be very difficult, as discussed in \cite{saari2023paraproducts}.

\textbf{Acknowledgments.}
F.Y-H. Lin is supported by the DAAD Graduate School Scholarship Programme - 57572629 and the Deutsche Forschungsgemeinschaft (DFG, German 
 Research Foundation) under Germany’s Excellence Strategy–EXC-2047/1–3
 90685813 as well as SFB 1060. This project was initiated and developed during Jiao Chen and Martin Hsu's visit to Bonn. The authors are grateful to Prof. Christoph Thiele for his generous hospitality and countless inspiring insights.

\section{Overview of the proof}

 For the proof of Theorem \ref{mainthm}, fix  $2<p_{1},p_{2},p_{3}<\infty$ with  $\frac{1}{p_{1}}+\frac{1}{p_{2}}+\frac{1}{p_{3}}=1$.
Fix also  \(0\leq \theta_{0}<\frac \pi 6\). Fix $s>1$. 
Let $N\in \N$.
For any quantities $A$, $B$ depending on these and possibly further  parameters, which most prominently 
will be $\Gamma$, $m$, we will write
\(A\lesssim B\) whenever \(A\leq C B\)
for some number $C$ depending on $p_1,p_2,p_3,\theta_{0},s,N$ only but not on the parameters. Analogously, we write \(A\gtrsim B\) whenever \(B\lesssim A\). If in particular, \(A\lesssim B\) and \(A\gtrsim B\) simultaneously we write \(A\sim B\).
 
For $(S,\mathcal{A},\mu)$ a measure space and $f$ a measurable function on this space, the $L^{p}$ norm of $f$ will be expressed as
\begin{equation*}
   \|f\|_{L^{p}_{\mu}(S)}=\|f(x)\|_{L^{p}_{\mu (x)}(S)}:=\left(\int_{S}|f|^{p}(x)d\mu(x)  \right)^{\frac{1}{p}}.
\end{equation*}
Furthermore, if $\mu (S)<\infty$, we define the average $L^{p}$ norm of $f$ as
\begin{equation*}
    \|f\|_{\L^{p}_{\mu}(S)}=\|f(x)\|_{\L^{p}_{\mu (x)}(S)}:=\left(\frac{1}{\mu (S)}\int_{S}|f|^{p}(x)d\mu(x)  \right)^{\frac{1}{p}}=\left(\fint_{S}|f|^{p}(x)d\mu(x)  \right)^{\frac{1}{p}}.
\end{equation*}
If it's clear from the context that the integration is over a space $V$ isomorphic to an $n$-dimensional Euclidean space with the usual $n$-dimensional Hausdorff measure, we simply write   $\|f(x)\|_{L^{p}_{x}(V)}$ instead of $\|f(x)\|_{L^{p}_{\mathcal{H}^{n}(x)}(V)}$.

% Let $\Gamma_{\iota}$ and $\Gamma$ be given as in the Theorem \ref{mainthm}.
% Let $m$ be a multiplier as in Theorem \ref{mainthm} and let $\Lambda_m$ as in 
% \eqref{eq_trilinear_form}.
Theorem \ref{mainthm} will be proven in Section \ref{proofmainthm} by reducing
to Proposition \ref{boundmodelform} which states a bound of a model form.

We write $\valpha=(\alpha_{1},\alpha_{2},\alpha_{3})$ for a typical element in $\R^3$ and $\vbeta=(\beta_{1},\beta_{2},\beta_{3})$ for a typical element on $V$.
Define $\mu$ a measure on $V$ which assigns zero measure to $\Gamma$ and has density
\[d\mu (\vbeta):=\frac{d\mathcal{H}^{2}(\vbeta)}{d_{\Gamma}(\vbeta)^{2}}\]
on $V\setminus \Gamma$. Define a measure on $\mathbb{R}\times V$ by
\begin{equation*}
    d\nu (\alpha,\vbeta):=d\alpha\otimes d\mu (\vbeta).
\end{equation*}
Let ${\Proj}_V$ be the orthogonal projection from $\R^3$ 
  onto $V$. Define a smooth function on $\mathbb{R}$
\begin{equation*}
    \widehat{\varphi}:={\Dil}^{\infty}_{2\varepsilon}\widetilde{\eta},
\end{equation*}
which is constant one in $B_{\frac{2\varepsilon}{10}}(0)$ and supported on $B_{\frac{4\varepsilon}{10}}(0)$. The number \(\varepsilon\) is a small constant which only depends on \(\theta_0\). The specific value of \(\varepsilon\) will be determined in Section \ref{secbessel}.

\begin{proposition}[Bound of the model form]\label{boundmodelform}
  Let $K:\R^3\times V\to \mathbb{C}$ be a continuous function  satisfying
  \eqref{kab} and \eqref{kernelcond} below.
  For all $\valpha\in \R^3$ and $\vbeta\in V$,
  \begin{equation}\label{kab}
  K(\valpha, \vbeta)=K({\Proj}_V\valpha, \vbeta).
  \end{equation}    For all  $\vbeta\in V $, $s>1$,
\begin{equation}\label{kernelcond}
    \left\|\left(1+|d_{\Gamma}(\vbeta)\valpha|^{2}\right)^{\frac{s}{2}}\cdot {K}(\valpha, \vbeta)\right\|_{L^{2}_{\valpha}(V)}\lesssim d_{\Gamma}(\vbeta).
\end{equation}
     Then for all Schwartz functions  
    $f_{1},f_{2},f_{3}$ on $\mathbb{R}$, we have the bound
\begin{equation}\label{modelbound}
    \left|\int_{V}\int_{\mathbb{R}^{3}}K(\valpha, \vbeta)\cdot \prod_{j=1}^{3}\left((\Mod_{\beta_{j}}\Dil_{d_{\Gamma}(\vbeta)^{-1}}^{1}\varphi) \ast f_{j}\right)(\alpha_{j})d\valpha d\mu (\vbeta)\right|
   \lesssim \prod_{j=1}^3 \|f_j\|_{p_j}.
\end{equation}

\end{proposition}

Proposition \ref{boundmodelform} is proven in Section \ref{sec_proofmodelbound}
by first reducing to the special case $N=1$.
Fix from now on $K$ as in Proposition \ref{boundmodelform}.
% \begin{proposition}
%     For every \(f_1,f_2,f_3\in\mathcal{S}\br{\R}\), for all continuous function $K:\R^3\times V\to \mathbb{C}$ such that for all 
%      $\vbeta\in V $ 
% \begin{equation}\label{kernelcond}
%     \left\|\left(1+|d_{\Gamma}(\vbeta)\valpha|^{2}\right)^{\frac{s}{2}}\cdot K(\valpha, \vbeta)\right\|_{L^{2}_{\valpha}(V)}\lesssim 1,
% \end{equation}
% % \(K\) e the \eqref
%     \begin{equation}
%         \Lambda (f_{1},f_{2},f_{3})=\int_{V}\int_{\mathbb{R}^{3}}d_{\Gamma}(\vbeta)^{\frac{1}{2}}K(\valpha, \vbeta)\cdot \prod_{j=1}^{3}(\varphi_{\vbeta ,j}\ast f_{j})(\alpha_{j})d\valpha d\mathcal{H}^{2}(\vbeta)
%     \end{equation}
% \end{proposition}
Define $\delta_{0}:=\frac{\sqrt{6}}{3}\operatorname{cos}(\theta_{0}+\frac{\pi}{3})$.

\begin{figure}
    \centering
\includegraphics[width=0.6\textwidth,height=8cm]{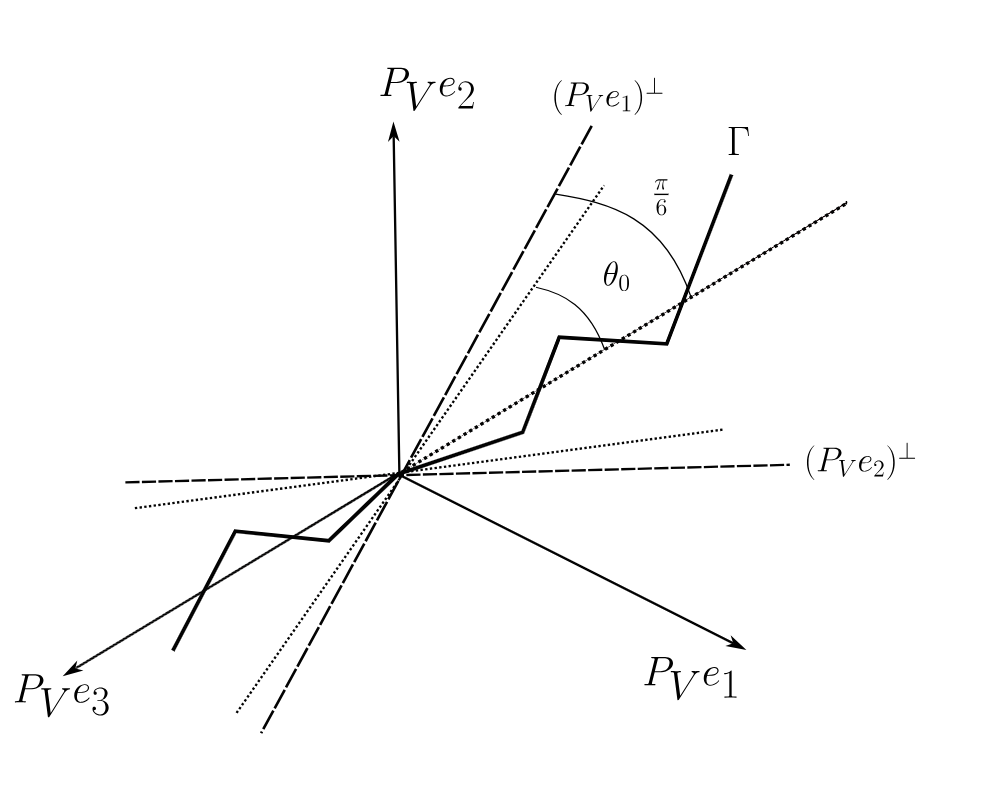}
    \caption{Lemma \ref{lemawayfromortho} explains that $\Gamma$ is also away from all the three degenerate directions.}
    \label{fig:enter-label}
\end{figure}

\begin{lemma}\label{lemawayfromortho}
    \noindent Assume $N=1$, for all $j=1,2,3$, we have for any $\vgamma,\vgamma'\in \Gamma$,
 \begin{equation}\label{singcond}
        |\ang{\vgamma-\vgamma', e_j}|=|\ang{\vgamma-\vgamma', P_{V}e_j}|\geq \delta_{0} \abs{\vgamma-\vgamma'}.
    \end{equation} 
\end{lemma}
\begin{proof}
 The equality on the left of \eqref{singcond} is true because $\vgamma-\vgamma'$ belongs to $V$. Notice from \eqref{stay_in_one_cond} that the angle between $\vgamma-\vgamma'$ and $P_{V}e_{j_{0}}$ is $\theta_{0}<\frac{\pi}{6}$. For $j\neq j_{0}$, the angle between the orthogonal 
 complement of $P_{V}e_{j}$ and $P_{V}e_{j_{0}}$ is $\frac{\pi}{6}$. Hence the angle between $\vgamma-\vgamma'$ and the orthogonal complement of $P_{V}e_{j}$ is a least $\frac{\pi}{6}-\theta_{0}$, and the angle between $\vgamma-\vgamma'$ and $P_{V}e_{j}$ is at most $\theta_{0}+\frac{\pi}{3}$. Then since $|P_{V}e_{j_{0}}|=\frac{\sqrt{6}}{3}$, we have \eqref{singcond}.
\end{proof}

Let $\theta_{1}=\frac{\pi}{18}-\frac{\theta_{0}}{3}$ be a fixed constant such that $\frac{\pi}{3}+\theta_{0}+\theta_{1}<\frac{\pi}{2}$ and $\sin\theta_{1}<\frac{\sqrt{6}}{3}\cos(\frac{\pi}{3}+\theta_{0}+\theta_{1})$. Let $\delta_{1}=\sin \theta_{1}$ and $\delta_{2}= \frac{\sqrt{6}}{3}\cos (\frac{\pi}{3}+\theta_{0}+\theta_{1})$.

For $\vgamma\in \Gamma$, $t\geq 0$, and $j\in \{1,2,3\}$ define the sets, $W_{\vgamma ,t}$, and $U_{\vgamma}^{j}$ as follow
\begin{equation}\label{defwgt}
  W_{\vgamma,t}:=\{\vbeta \in V: t \leq |\vbeta-\vgamma| \leq \frac{1}{\delta_{1}}d_\Gamma\br{\vbeta}\}.
% \leq \abs{\vbeta-\vgamma}. 
\end{equation}
%Let $W_{\vgamma, t}$ be the set of $\vbeta\in W_\vgamma$ such that 
\begin{equation}\label{Uangle}
    U_{\vgamma}^{j}:=\left\{\vbeta\in V : |\langle \vbeta-\vgamma ,e_{j}\rangle |\leq  \delta_{2}|\vbeta-\vgamma| \right\}.
\end{equation}

Let $\mathcal{I}$ be the collection of all intervals in $\mathbb{R}$. Let $ \mathbf{T}$ be the set of all pairs $T=(I,\vgamma)$ with $I\in \mathcal{I}$ and $\vgamma\in \Gamma$. For such $T$, we associate a region $D_{T}:=I\times W_{\vgamma,1/|I|}$. We choose the letter $T$ here because parts of the literature \cite{do2015lp} refer to closely related objects as tents. 
We define for $j\in \left\{1,2,3\right\}$
and a function $f$ on $\R$ the function $F_jf$ on $\mathbb{R}\times V$ by
\begin{equation}\label{definefj}
(F_{j}f)(\alpha_{j},\vbeta):=\left((\Mod_{\beta_{j}}\Dil_{d_{\Gamma}(\vbeta)^{-1}}^{1}\varphi) \ast f\right)(\alpha_{j}).
\end{equation}

For a set $I \times W_{\gamma ,t}\subseteq \mathbb{R}\times V$, an index $j\in \{1,2,3\}$, and a function $F$ on $\mathbb{R}\times V$, we define a local size $S^j$ of $F$ associated with $I \times W_{\gamma ,t}$ 
\begin{equation}\label{sizedef}
    \|F\|_{S^{j}(I,\vgamma,t)}:=|I|^{-\frac{1}{2}}\|F\|_{L^{2}_{\nu}(I\times (W_{\vgamma, t}\setminus U_\vgamma^j))}\vee \nrm{F}_{L^{\infty}(I \times W_{\gamma ,t})}
\end{equation}
\begin{equation*}
    =\left(\frac{1}{|I|}\int_{I\times\br{W_{\vgamma, t}\setminus U_\vgamma^j}}|F\br{\alpha,\vbeta}|^{2} d\alpha d\mu (\vbeta)\right)^{\frac{1}{2}}\vee \nrm{F}_{L^{\infty}(I \times W_{\gamma ,t})}.
\end{equation*}
If in particular $t=\frac{1}{|I|}$, we write 
\begin{equation*}
    \|F\|_{S^{j}(I,\vgamma)}:=\|F\|_{S^{j}(I,\vgamma,t)}.
    %\|F\|_{L^{2}_{avg}(I)\otimes L^{2}_{d\mu}(W_{\vgamma, 1/|I|}\setminus U_\vgamma^j)}\vee \nrm{F}_{L^{\infty}(r_T)}
\end{equation*}
We also define a global size
\begin{equation*}
    \|F\|_{S^{j}}:=\underset{{I\in \mathcal{I}, \vgamma\in \Gamma}}{\operatorname{sup}}\|F\|_{S^{j}(I,\vgamma)}.
\end{equation*}
The model form is estimated first on certain regions associated to tents in Proposition \ref{propTentest}. To obtain the sharp regularity $s$ in the form of condition \eqref{Hormcond}, we prove Proposition \ref{propTentest} in Section \ref{sectentest} by splitting the frequency region into small and large scale, then performing different estimates on the respective pieces.

\begin{proposition}[Tent Estimate]\label{propTentest}
Assume $N=1$. Let $i\in \{1,2,3\}$. Let $I\in \mathcal{I}$ and $\vgamma \in \Gamma$. Then we have the inequality 
\begin{equation}\label{eq216}
    \left\|K(\valpha, \vbeta)\cdot \prod_{j=1}^{3}(F_{j}f_{j})(\alpha_{j},\vbeta)\right\|_{L^1_{\valpha, \mu(\vbeta)}((I\ve_{i}\oplus \ve_{i}^{\perp})\times W_{\vgamma,1/|I|})}\lesssim |I|\cdot \prod_{j=1}^{3}\|F_{j}f_{j}\|_{S^{j}}.
\end{equation}
\end{proposition}
Naturally, we aim to control the right-hand side of \eqref{eq216}. On the one hand, a simple \(L^\infty\) bound can be obtained as follows.
\begin{proposition}[Global Estimate]\label{prop_Linfty_emb}
    Assume $N=1$. Given $f\in L^2\br{\R}\cap L^\infty\br{\R}$, we have
    \begin{equation}
        \nrm{F_{j}f}_{S^j}\lesssim \nrm{f}_{L^\infty}.
    \end{equation}
\end{proposition}
On the other hand, we aim to obtain a certain \(L^2\) estimate to serve as the other endpoint and perform an interpolation argument. At this stage, the main difficulty of proving Proposition \ref{boundmodelform} is to be efficient in summing all the pieces on the left-hand side of \eqref{eq216} in Proposition \ref{propTentest}. Therefore, we must derive certain orthogonality among objects associated with tents. To address the orthogonality issue, we build up Proposition \ref{proptentgeo} to treat the distribution and the interaction among Whitney balls associated with a Lipschitz singularity.
  
% To deal with the issue how Whitney balls corresponding to a Lipschitz singularity distribute and interact with each other in this continuous model, we build up certain geometry in Proposition \ref{proptentgeo}.

\begin{proposition}[Geometry of Tents]\label{proptentgeo} Assume $N=1$.
Let $1\le j\le 3$.
We define 
\begin{equation}\label{rhorange}
    \rho:=\frac{\delta_{2}-\delta_1}{1+\delta_1}.
\end{equation}
Let $\vgamma,\vgamma'$ be two distinct points on $\Gamma$ and $t>0$.

\noindent (1)
Let $ \vgamma''$ be another point on $\Gamma$ satisfying
\begin{equation}\label{meshcond}
    \vgamma_j\leq {\vgamma''}_j\leq {\vgamma'}_j\leq  \vgamma_j+\delta_{0}(1-\delta_1) t.
\end{equation}
Then
\begin{equation}
W_{\vgamma'',t}\subseteq W_{\vgamma,\delta_1 t}\cup W_{\vgamma',\delta_1 t}.
\end{equation}
(2)
Given two points
\begin{equation}\label{eq221}
\vbeta \in W_{\vgamma,t}\setminus U_{\vgamma}^{j},\quad \vbeta' \in W_{\vgamma',0}\setminus W_{\vgamma ,\delta_1 t}
\end{equation} with $\beta_{j}<\vgamma_{j}<{\vgamma'}_{j}$, then 
\begin{equation}\label{eq222}
    B_{{ \rho }d_{\Gamma}(\vbeta)}(\beta_{j})\cap  B_{ \rho d_{\Gamma}(\vbeta^{'})}(\beta_{j}^{'}
    )=\varnothing.
\end{equation}
\end{proposition}
\begin{figure}[h]
    \centering
\includegraphics[width=0.6\textwidth,height=8cm]{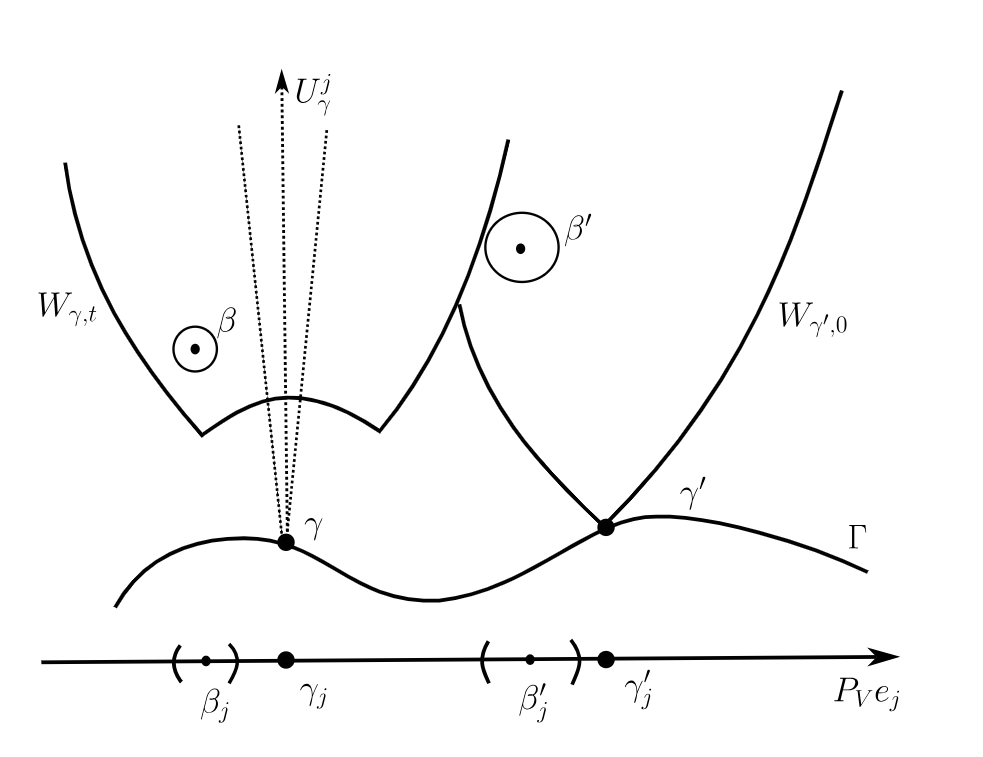}
    \caption{In \eqref{eq221} and \eqref{eq222}, we describe certain orthogonality.}
    \label{fig:enter-label}
\end{figure}
The geometry in Proposition \ref{proptentgeo} serves as a base for the two algorithms introduced in Proposition \ref{prop_select} and \ref{prop_selectL2}. Essentially, both algorithms extract collections of countable tents with desired geometric properties from collections of uncountable tents.
\begin{proposition}[Selection Algorithm, $L^{\infty}$ Component]\label{prop_select}
    Let \(\Omega\subseteq \R\times \br{V\setminus\Gamma}\) be compact and \(\lambda>0\) be a threshold. For \(j=1,2,3\) and $f\in L^2\br{\R}\cap L^\infty\br{\R}$, we have the following:
    There is a countable collection of tents $\mathbf{T}$ and a countable collection of points $\P$ in $\Omega$ that satisfy the following properties:
    \begin{itemize}
        \item Covering: 
        \begin{equation}\label{select1-1}
             \P\subseteq \Omega\cap \left\vert F_{j}f\right\vert^{-1}(\lambda ,2\lambda ]\subseteq \bigcup_{T=(I,\vgamma)\in \mathbf{T}}D_{T}.
        \end{equation}
        
        \item Estimate:
        \begin{equation}\label{select1-2}
            \sum_{T=(I,\vgamma)\in \mathbf{T}}\abs{I}
            \sim
            \sum_{(\alpha,\vbeta) \in \P}d_\Gamma\br{\vbeta}^{-1}
            \lesssim
            \sum_{(\alpha,\vbeta) \in \P}
            \frac{\abs{F_j f\br{\alpha,\vbeta}}^2}{d_\Gamma\br{\vbeta}\lambda^2}.
        \end{equation}
        
        \item Orthogonality: For distinct \(\br{\alpha,\vbeta},\br{\alpha',\vbeta'}\in  \P\), at least one of the following statements holds.
        \begin{equation}\label{select1-31}
             \abs{\alpha-\alpha'}\geq  2 \br{d_\Gamma\br{\vbeta}^{-1}+d_\Gamma\br{\vbeta'}^{-1}}.
        \end{equation}
        \begin{equation}\label{select1-32}
             \abs{\beta_j-\beta_j'}\geq  \rho  \br{d_\Gamma\br{\vbeta}+d_\Gamma\br{\vbeta'}}.
        \end{equation}
    \end{itemize}
% If further that \(\Omega\subset \abs{F_jf}^{-1}\Br{0,\lambda}\), we have the following lemma:
\end{proposition}
To state the next proposition, we introduce two auxiliary sets.
Given \(\vgamma\in \Gamma\) and \(t\geq 0\), we define
\begin{equation}
    \left(W_{\vgamma,t}\setminus U_{\vgamma}^{j}\right)^{<j}:=\{\vbeta \in V: \beta_{j}<\gamma_{j}\}\cap \left( W_{\vgamma,t}\setminus U_\vgamma^j \right)
\end{equation}
and
\begin{equation}
    \left(W_{\vgamma,t}\setminus U_{\vgamma}^{j}\right)^{>j}:=\{\vbeta \in V: \beta_{j}>\gamma_{j}\}\cap \left( W_{\vgamma,t}\setminus U_\vgamma^j \right).
\end{equation}

\begin{proposition}[Selection Algorithm, $L^{2}$ Component]\label{prop_selectL2}
 Let \(\Omega\subseteq \R\times \br{V\setminus\Gamma}\) be compact and \(\lambda>0\) a threshold. For \(j=1,2,3\) and $f\in L^2\br{\R}\cap L^\infty\br{\R}$, we have the following:
    There is a countable collection $\mathbf{T}$ of tents and a countable collection $\mathbf{S}$ of the form \(\br{I,\vgamma, S}\) with $(I,\vgamma)$ a tent and $S$ a measurable subset of $ \Omega\cap I\times \left(W_{\vgamma,1/\abs{I}}\setminus U_{\vgamma}^{j}\right)^{<j}$  that satisfy the following properties:
    \begin{itemize}
        \item Covering: For any tent \( (I, \vgamma) \),
        \begin{equation}\label{select2-1}
            |I|^{-\frac{1}{2}}\nrm{
                1_{\Omega\setminus\bigcup_{T\in \T}D_{T}} F_jf
            }_{L^2_{\nu}\br{ I\times (W_{\vgamma,1/|I|}\setminus U_{\vgamma}^{j})^{<j}}}\leq \frac{\lambda}{\sqrt{2}}
        \end{equation}
        
        \item Estimate:
        \begin{equation}\label{select2-2}
            \sum_{(I,\vgamma )\in \T}\abs{I}
            \sim
            \sum_{\br{I,\vgamma,S}\in\S}
            \abs{I}
            \lesssim
            \sum_{\br{I,\vgamma,S}\in\S}
            \frac{\nrm{F_j f}_{L^2_{\nu}\br{S}}^2}{\lambda^2}
        \end{equation}
        In particular, for \(\br{I,\vgamma,S}\in \S\), we have \(\abs{I}^{-\frac{1}{2}}\nrm{F_j f }_{L^2_{\nu}\br{S}}\gtrsim \lambda\).
        
        \item Orthogonality: For distinct \(\br{I,\vgamma,S},\br{I',\vgamma',S'}\in \S\) with \(\gamma_j\leq \gamma_j'\), any pair of \(\br{\alpha,\vbeta}\in S\) and \(\br{\alpha',\vbeta'}\in S'\) satisfies at least one of the following.
        \begin{equation}\label{select2-31}
            \abs{\alpha-\alpha'}\geq  2|I|.
        \end{equation}
        \begin{equation}\label{select2-32}
             \abs{\beta_j-\beta_j'}\geq  \rho  \br{d_\Gamma\br{\vbeta}+d_\Gamma\br{\vbeta'}}.
        \end{equation}
    \end{itemize}
Symmetrically, the proposition holds for  $\left(W_{\vgamma,\frac{1}{\abs{I}}}\setminus U_{\vgamma}^{j}\right)^{<j}$ replaced by \\
$\left(W_{\vgamma,\frac{1}{\abs{I}}}\setminus U_{\vgamma}^{j}\right)^{>j}$.
\end{proposition}
Once Propositions \ref{prop_select} and \ref{prop_selectL2} guarantee the existence of well-behaved configurations of tents, the proof of the next proposition mainly follows the same line of argument as in \cite{do2015lp} and \cite{muscalu2013classical}. For completeness, we include the proof in section \ref{secbessel}.
\begin{proposition}[Bessel Type Estimate]\label{prop_L2_emb}
    Assume $N=1$. Given a compact set \(\Omega\) in \(\R\times \br{V\setminus \Gamma}\), a function $f\in L^2\br{\R}\cap L^\infty\br{\R}$, and \(\lambda>0\), there is a countable collection of tents $\T$
    such that
    \begin{equation}\label{eq_Tlarge_est}
        \sum_{(I,\vgamma )\in \T}\abs{I}\lesssim \frac{\nrm{f}_{L^2}^2}{\lambda^2}
    \end{equation}
        and
    \begin{equation}\label{eq_Omegasmall_con}
        \nrm{ 1_{\Omega\setminus\bigcup_{T\in \T}D_{T}}\cdot F_{j}f}_{S^j}\leq \lambda.
    \end{equation}
\end{proposition}

\section{Proof of Theorem \ref{mainthm}}\label{proofmainthm}
Define a smooth function on $V$
\begin{equation*}
\rchi(\vbeta):={\Dil}^{\infty}_{\varepsilon}\widetilde{\eta}(|\vbeta|).
\end{equation*}
%which is constant one in $B_{\frac{\varepsilon}{10}}(0)$ and supported in $B_{\frac{2\varepsilon}{10}}(0)$.
 For $\bm{\vbeta} \in V\setminus \Gamma$,  define the bump function adapted to position $\bm{\vbeta}$ by
\begin{equation*}
\rchi_{\vbeta}:={\Tr}_{\vbeta}{\Dil}^{\infty}_{d_{\Gamma}(\vbeta)}\rchi \, .
\end{equation*}
Define the normalized bump function as
\begin{equation*}
\widetilde{\rchi}_{\vbeta}:=X_{\Gamma}^{-1}\cdot \rchi_{\vbeta},
\end{equation*}
where
\begin{equation*}
    X_{\Gamma}:=    %\int_{V}\rchi_{\vbeta}\frac{d\mathcal{H}^{2}(\vbeta)}{d_{\Gamma}(\vbeta)^{2}}=
    \int_{V}\rchi_{\vbeta}d\mu (\vbeta).
\end{equation*}
Then we have the identity
\begin{equation*}
    1_{V\setminus \Gamma}=\int_{V}\widetilde{\rchi}_{\vbeta}d\mu (\vbeta).
\end{equation*}
As $\rchi$ is supported in $B_{\frac{2\varepsilon}{10} }(0)$ and $\widehat{\varphi}$ is constant one on $B_{\frac{2\varepsilon}{10}}(0)$,
we have
\begin{equation*}
    1_{V\setminus \Gamma}(\vxi)=\int_{V}\widetilde{\rchi}_{\vbeta}(\vxi)\prod_{j=1}^{3}\left({\Tr}_{\beta_{j}}{\Dil}^{\infty}_{d_{\Gamma}(\vbeta)}\widehat{\varphi}\right)(\xi_{j})d\mu (\vbeta).
\end{equation*}
Inserting this into \eqref{eq_trilinear_form}, in our case $n=2$, we obtain for $\Lambda (f_{1},f_{2},f_{3})$ the expression
\begin{equation}\label{eqthreesix}
    \begin{aligned}
        %\Lambda (f_{1},f_{2},f_{3})
        %&=
%\int_{V}m(\vxi)\prod_{j=1}^{3}\widehat{f}_{j}(\xi_{j})d\mathcal{H}^{2}(\vxi)\\ &=
\int_{V}\int_{V}\left(m(\vxi)\widetilde{\rchi}_{\vbeta}(\vxi)\right)\cdot \prod_{j=1}^{3}\left(\left({\Tr}_{\beta_{j}}{\Dil}^{\infty}_{d_{\Gamma}(\vbeta)}\widehat{\varphi}\right)(\xi_{j})\widehat{f}_{j}(\xi_{j}) \right)d\mathcal{H}^{2}(\vxi)d\mu(\vbeta).
    \end{aligned}
\end{equation}
Let $m_{\vbeta}=m\cdot \widetilde{\chi}_{\vbeta}$ and $\mathcal{F}f=\widehat{f}$ be the Fourier transform of $f$. We define
\begin{equation}
    K\br{\valpha,\vbeta}:=\mathcal{F}(m_{\vbeta}\cdot d\mathcal{H}^{2}_{V})(\valpha)
    =\int_{V}m_{\vbeta}(\vxi)e^{-2\pi i\valpha\cdot \vxi}d\mathcal{H}^{2}(\vxi).
\end{equation}
Note that $K(\valpha, \vbeta)$ satisfies the invariance property \eqref{kab}.
Applying Plancherel to the inner integral in \eqref{eqthreesix}, we obtain 
\begin{equation*}
    \Lambda (f_{1},f_{2},f_{3})=\int_{V}\int_{\mathbb{R}^{3}}K\br{\valpha,\vbeta}\cdot \prod_{j=1}^{3}\left(({\Mod}_{\beta_{j}}{\Dil}_{d_{\Gamma}(\vbeta)^{-1}}^{1}\varphi) \ast f_{j}\right)(\alpha_{j})d\valpha d\mu(\vbeta)
\end{equation*}
\begin{equation}
   = \int_{V}\int_{\mathbb{R}^{3}}K\br{\valpha,\vbeta}\cdot \prod_{j=1}^{3}(F_{j}f_j)(\alpha_{j},\vbeta)d\valpha d\mu(\vbeta)\, ,
\end{equation}
where we use the notation $F_j$ as defined in \eqref{definefj}.

We verify the kernel condition \eqref{kernelcond} for $K$. Let $\vbeta \in V\setminus \Gamma$ and $s>1$.
%\begin{equation}
%    \nrm{\br{1+\abs{d_\Gamma\br{\vbeta}\valpha}^2}^{\frac{s}{2}}\cdot K\br{\valpha,\vbeta}}_{L^2_{d\mathcal{H}^{2}(\valpha)}\br{V}}\lesssim d_{\Gamma}(\vbeta).
%\end{equation}
Expanding the kernel, we observe
\begin{equation}\label{eqthreten}
 \left\|d_{\Gamma}(\vbeta)^{-1}\left(1+|d_{\Gamma}(\vbeta)\valpha|^{2}\right)^{\frac{s}{2}}\cdot K\br{\valpha,\vbeta}\right\|_{L^{2}_{\valpha }(V)}
 \end{equation}
 \begin{equation*}
       =
        \left\|d_{\Gamma}(\vbeta)^{-1}\left(1+|d_{\Gamma}(\vbeta)\valpha|^{2}\right)^{\frac{s}{2}}\cdot \mathcal{F}\left( m \cdot \left( {\Tr}_{\vbeta}{\Dil}^{\infty}_{d_{\Gamma}(\vbeta)}\rchi \right)\cdot X_{\Gamma}^{-1} \cdot d\mathcal{H}^{2} \right)(\valpha)\right\|_{L^{2}_{\valpha}(V)}\, .
\end{equation*}     
Applying the $L^2$ isometry  ${\Dil}_{d_{\Gamma}(\vbeta)}^{2}{\Mod}_{-\vbeta}$ on the function inside the $L^{2}$ norm and distributing powers of $d_{\Gamma}(\vbeta)$
 equates \eqref{eqthreten} with 
\begin{equation*}
\left\|\left(1+|\valpha|^{2}\right)^{\frac{s}{2}}\cdot \mathcal{F}\left(\left( {\Dil}^{\infty}_{d_{\Gamma}(\vbeta)^{-1}}{\Tr}_{-\vbeta}m \right) \cdot \rchi  \cdot \left( {\Dil}^{\infty}_{d_{\Gamma}(\vbeta)^{-1}}{\Tr}_{-\vbeta}X_{\Gamma}^{-1} \right) \cdot d\mathcal{H}^{2} \right)(\valpha)\right\|_{L^{2}_{\valpha}(V)}.
\end{equation*}
Applying the definition of $H^{s}(V)$  and the fact that $H^{s}(V)$ is a Banach algebra when $s>1$, we then estimate \eqref{eqthreten} by        \begin{equation*}
        \left\| \left( {\Dil}^{\infty}_{d_{\Gamma}(\vbeta)^{-1}}{\Tr}_{-\vbeta}m \right)\cdot \Phi \cdot  \left( {\Dil}^{\infty}_{d_{\Gamma}(\vbeta)^{-1}}{\Tr}_{-\vbeta}X_{\Gamma}^{-1} \right)\cdot 
     \rchi  \right\|_{H^{s}(V)}
\end{equation*}
\begin{equation}\label{eq313}
        \lesssim \left\|  \left( \mathrm{Dil}^{\infty}_{d_{\Gamma}(\vbeta)^{-1}}\mathrm{Tr}_{-\vbeta}m \right)\cdot \Phi \right\|_{H^{s}(V)}\cdot  \left\| \left( \mathrm{Dil}^{\infty}_{d_{\Gamma}(\vbeta)^{-1}}\mathrm{Tr}_{-\vbeta}X_{\Gamma}^{-1} \right)\cdot  \rchi  \right\|_{H^{s}(V)}
\end{equation}

 The first factor in \eqref{eq313} is bounded by $1$ by \eqref{Hormcond}. We introduce three lemmas to prove the bound for the second factor.

\begin{lemma}\label{betacomparable}
   For $\vbeta_1,\vbeta_2 \in V\setminus \Gamma$, we have
   \begin{equation}\label{eqbc315}
   d_\Gamma (\vbeta_1)\le d_\Gamma(\vbeta_2)+|\vbeta_1-\vbeta_2|
   \end{equation}
 \end{lemma}
\begin{proof}
    For all $\epsilon>0$, there exists $\vgamma_{2}\in \Gamma$ such that
    \begin{equation}\label{eqbc316}
        |\vgamma_{2}-\vbeta_{2}|<d_{\Gamma}(\vbeta)+\epsilon.
    \end{equation}
    By the triangle inequality,
\begin{equation*}
    d_{\Gamma}(\vbeta_{1})\leq |\vgamma_{2}-\vbeta_{1}|\leq
    |\vgamma_{2}-\vbeta_{2}|+ |\vbeta_{1}-\vbeta_{2}|<d_{\Gamma}(\vbeta_{2})+|\vbeta_{1}-\vbeta_{2}|+\epsilon
\end{equation*}
Since \eqref{eqbc316} holds for all $\varepsilon>0$, we obtain \eqref{eqbc315}.
\end{proof}

\begin{lemma}\label{lemlowbdd}
    For $x\in V\setminus \Gamma$, then
\begin{equation}\label{rchilowerbdd}
   |X_{\Gamma}(x)|\gtrsim 1. 
\end{equation}
\end{lemma}

\begin{proof}
    For $\vbeta \in B_{r}(x)$, where $r=\frac{\varepsilon}{20}d_{\Gamma}(x)$, by Lemma \ref{betacomparable}, we have
\begin{equation*}
    d_{\Gamma}(x)\leq d_{\Gamma}(\vbeta)+|\vbeta-x|\leq d_{\Gamma}(\vbeta)+\frac{\varepsilon}{20}d_{\Gamma}(x).
\end{equation*}
Hence $d_{\Gamma}(x)$ can be dominated by a constant times $d_{\Gamma}(\vbeta)$.
\begin{equation*}
    d_{\Gamma}(x)<\frac{1}{1-\varepsilon/20}d_{\Gamma}(\vbeta).
\end{equation*}
Then
\begin{equation*}
    |\vbeta-x|\leq \frac{\varepsilon}{20}d_{\Gamma}(x)\leq  \frac{\varepsilon}{10}d_{\Gamma}(\vbeta).
\end{equation*}
 That is $\rchi_{\vbeta}(x)=1$. On the other hand, for $\vbeta \in B_{r}(x)$,
 \begin{equation*}
     d_{\Gamma}(\vbeta)\leq d_{\Gamma}(x)+|\vbeta -x|<(1+\frac{\varepsilon}{20})d_{\Gamma}(x).
 \end{equation*}
 Therefore, we obtain a lower bound of $X_{\Gamma}(x)$.
 \begin{equation*}
     X_{\Gamma}(x)=\int_{V}\rchi_{\vbeta}\frac{d\mathcal{H}^{2}(\vbeta)}{d_{\Gamma}(\vbeta)^{2}}\geq \int_{B_{r}(x)}(1+\frac{\varepsilon}{20})^{-2}d_{\Gamma}(x)^{-2}d\mathcal{H}^{2}(\vbeta)\gtrsim 1.
 \end{equation*}
\end{proof}

\begin{lemma}\label{lemderbdd}
    For $x\in V\setminus \Gamma$, and multi-index $\alpha$,
\begin{equation}\label{dericontr}
    \left|(\partial^{\alpha} X_{\Gamma}\right) (x)|\lesssim d_{\Gamma}(x)^{-|\alpha|}.
\end{equation}
\end{lemma}
\begin{proof}
Assume $\vbeta \in V\setminus \Gamma$ satisfies
\begin{equation}\label{betarange}
    |\vbeta -x|<\frac{2\varepsilon}{10}d_{\Gamma}(\vbeta).
\end{equation}
Then
\begin{equation*}
    d_{\Gamma}(\vbeta)\leq d_{\Gamma}(x)+|\vbeta -x|\leq  d_{\Gamma}(x)+\frac{2\varepsilon}{10}d_{\Gamma}(\vbeta).
\end{equation*}
Hence $d_{\Gamma}(\vbeta)$ can be dominated by a constant times $d_{\Gamma}(x)$.
\begin{equation*}
    d_{\Gamma}(\vbeta)\leq \frac{1}{1-2\varepsilon/10}d_{\Gamma}(x).
\end{equation*}
Then
\begin{equation*}
    |\vbeta -x|\leq \frac{2\varepsilon}{10-2\varepsilon}d_{\Gamma}(x)\leq \frac{2\varepsilon}{5}d_{\Gamma}(x).
\end{equation*}
Also note that for $\vbeta$ satisfies \eqref{betarange},
\begin{equation*}
    d_{\Gamma}(x)\leq d_{\Gamma}(\vbeta)+|\vbeta -x|\leq (1+\frac{2\varepsilon}{10})d_{\Gamma}(\vbeta).
\end{equation*}
Let $r=\frac{2\varepsilon}{5}d_{\Gamma}(x)$. We have an upper bound of derivatives of $X_{\Gamma}(x)$.
\begin{equation*}
      \left|(\partial^{\alpha} X_{\Gamma}\right) (x)|\leq \int_{B_{r}(x)}\|\partial^{\alpha}\rchi \|_{L^{\infty}}\cdot d_{\Gamma}(\vbeta)^{-|\alpha|}\frac{d\mathcal{H}^{2}(\vbeta)}{d_{\Gamma}(\vbeta)^{2}}\lesssim d_{\Gamma}(x)^{-|\alpha|}.
\end{equation*}
\end{proof}

Back to the estimate of the second factor of \eqref{eq313}. Let $A$ be the least integer larger than $s$. Denote $ \mathrm{Dil}^{\infty}_{d_{\Gamma}(\vbeta)^{-1}}\mathrm{Tr}_{-\vbeta}X_{\Gamma}$ as $X_{\Gamma,\vbeta}$
and note that by Lemma \ref{lemlowbdd} and Lemma \ref{lemderbdd}, we have $|X_{\Gamma ,\vbeta}^{-1}|\lesssim 1$ and $|\partial^{\alpha}X_{\Gamma ,\vbeta}|\lesssim 1$  on the support of $\rchi$. By chain rule and Leibniz rule, this implies

 \begin{equation*}
      \left\|X_{\Gamma,\vbeta}^{-1}\cdot \rchi \right\|_{H^{s}}\lesssim \left\|X_{\Gamma,\vbeta}^{-1}\cdot \rchi \right\|_{L^{2}}+ \sum_{|\alpha|=A} \left\|\partial^{\alpha} \left(X_{\Gamma,\vbeta}^{-1} \cdot \rchi \right) \right\|_{L^{2}}
\end{equation*}
 \begin{equation}\label{reg334}
    \lesssim 1+ \sum_{|\alpha|\le A} \left\|\partial^{\alpha} \left(X_{\Gamma,\vbeta}^{-1}  \right) \right\|_{L^{\infty}(B_{\frac{2\varepsilon}{10}}(0))}\lesssim 1.
\end{equation}
% \begin{equation}
%   \lesssim 1+ \sum_{|\alpha|\le A} \left\|\partial^{\alpha} \left(X_{\Gamma,\vbeta}  \right) \right\|_{L^{\infty}(B_{\frac{2\varepsilon}{10}}(0))}
%      \left\|X_{\Gamma,\vbeta}  ^{-|\alpha|} \right\|_{L^{\infty}(B_{\frac{2\varepsilon}{10}}(0))}\lesssim 1.
% \end{equation}
This completes the estimate of \eqref{eq313}.

%\section{Proof of Lemma \ref{subconelemma} and Lemma \ref{scalebddlemma}: Geometry of Cones}
%\input{geometry of cones}

%\section{Proof of Lemma \ref{strdis} and Lemma \ref{lem_sym_tent_cover}: Geometry of Tents}
%\input{geometry of tents}

\section{Proof of Proposition \ref{propTentest}: Tent Estimate}\label{sectentest}
Let $i$ be given, and let $i'$ and $i''$ be the indices in $\{1,2,3\}$ different from $i$. We decompose the kernel $K$ 
\begin{equation}\label{kerdec}
K=\sum_{k\ge 0} K1_{E_k}\end{equation}
where 
\begin{equation}\label{sigmatorus}
E_k:=
\left\{
    \begin{aligned}
    &\left\{
        \br{\valpha,\vbeta}\in \mathbb{R}^{3}\times V
    :
       \abs{d_\Gamma(\vbeta)\Proj_{V}\valpha}
        \in\left(2^{k-1},2^k\right]
    \right\}, & k\in \N. \\
    &\left\{
        \br{\valpha,\vbeta}\in \mathbb{R}^{3}\times V
    :
       \abs{d_\Gamma(\vbeta)\Proj_{V}\valpha}
        \leq 1
    \right\}, & k = 0.
    \end{aligned}
\right.
\end{equation}
Identity \eqref{kerdec} holds because $\left\{E_k\right\}_{k=0}^{\infty}$ partitions 
$\mathbb{R}^{3}\times V$. To show \eqref{propTentest}, it suffices to show for each $k\ge 0$
\begin{equation}\label{tentestk}
    \left\|K(\valpha, \vbeta)\cdot \prod_{j=1}^{3}(F_{j}f_{j})(\alpha_{i},\vbeta)\right\|_{L^{1}_{\valpha,\mu(\vbeta)}\br{
        E_k\cap (I\ve_{i}\oplus \ve_{i}^{\perp})\times W_{\vgamma,1/|I|}
        % ,d\mathcal{H}^{3}(\valpha)d\mathcal{H}^{2}(\vbeta)
    }}\!\!\!\lesssim\br{1+k}2^{k\br{1-s} }|I|\prod_{j=1}^{3}\|F_{j}f_{j}\|_{S^{j}}
\end{equation}
because summing the right-hand side over $k$ gives the desired result
since \(s>1\).

\begin{lemma}\label{lem41}
For all $k\ge 0$, $\br{\valpha,\vbeta}\in E_k$, and $j,j'\in \{1,2,3\}$
\begin{equation}\label{ksigma supp control}
     \abs{\alpha_j-\alpha_{j'}}\leq\frac{2^{k+1}}{d_\Gamma\br{\vbeta}}.
\end{equation}
\end{lemma}
\begin{proof}
  Define $c$ implicitly by the condition
\begin{equation*}
    (\alpha_{1},\alpha_{2},\alpha_{3})+c (1,1,1)\in V.
\end{equation*}
By construction, \eqref{sigmatorus} implies
\begin{equation*}
    |(\alpha_{1}+c,\alpha_{2}+c,\alpha_{3}+c)|\leq \frac{2^k}{d_{\Gamma}(\vbeta)} .
    % \Longrightarrow \sum_{i=1}^{3}|\alpha_{i}+c|^{2}\eqsim \frac{2^{2k}}{d_{\Gamma}(\vbeta)^{2}} 
\end{equation*}
The triangle inequality yields
\begin{equation*}
    |\alpha_{j'}-\alpha_{j}|=|(\alpha_{j'}+c)+(-\alpha_{j}-c)|\leq |\alpha_{j'}+c|+|\alpha_{j}+c|\leq \frac{2^{k+1}}{d_{\Gamma}(\vbeta)}.
\end{equation*}
\end{proof}

In the rest of this section, to simplify the notation, we denote $W_{\vgamma.\frac{\tau}{|I|}}$ as $W_{\tau}$. We apply Cauchy Schwarz in the integration over $\alpha_{i'}$ and $\alpha_{i''}$ to estimate the left-hand side of \eqref{tentestk}
by 
\begin{equation}\label{fourten}
     \int_{W_{1}}
            \int_{I}
                \|\br{1_{E_k}K}(\valpha,\vbeta)\|_{L^{2}_{\alpha_{i'},\alpha_{i''}}}
         \cdot \left\| \prod_{j=1}^{3}(F_{j}f_{j})(\alpha_{j},\vbeta)  \right\|_{L^{2}_{\alpha_{i'},\alpha_{i''}}
        \left(Q_{\alpha_{i}}^2\right)}d\alpha_{i}d\mathcal\mu(\vbeta),
\end{equation}
where we use Lemma \ref{lem41} to restrict the domain of the last $L^2$ norm
to $Q_{a_{i}}^{2}$ with
\begin{equation*}
    Q_{\alpha_{i}}=\left[\alpha_{i}-\frac{2^{k+1}}{d_{\Gamma}(\vbeta)},\;\alpha_{i}+\frac{2^{k+1}}{d_{\Gamma}(\vbeta)} \right].
\end{equation*}
As $K$ is constant on any fiber parallel to $(1,1,1)$, the first factor 
in \eqref{fourten} is a universal constant times
\begin{equation}\label{eq410}
    \|\br{1_{E_k}K}(\valpha, \vbeta)\|_{L^{2}_{\valpha}(V)}.
\end{equation}
We estimate \eqref{eq410} via the kernel condition \eqref{kernelcond} and obtain for 
\eqref{fourten} the bound
\begin{equation}\label{foureleven}
  \lesssim   \int_{W_{1}}
            \int_{I}
               2^{-k s}\cdot d_{\Gamma}(\vbeta)
         \cdot \left\| \prod_{j=1}^{3}(F_{j}f_{j})(\alpha_{j},\vbeta)  \right\|_{L^{2}_{\alpha_{i'},\alpha_{i''}}
         \left(Q^{2}_{\alpha_{i}}\right)}d\alpha_{i}d\mathcal\mu(\vbeta).
\end{equation}
Now we split $W_{1}$ as 
\[\left(W_{1}\setminus W_{2^{k}}\right)\cup W_{2^{k}}.\]
We split \eqref{foureleven} accordingly and estimate the
pieces separately. 

Starting with the first piece, estimating the triple product by its 
sup norm, and using that length of $Q_{\alpha_{i}}$ is $\frac{2^{k+2}}{d_{\Gamma}(\vbeta)}$, we estimate this piece by
\begin{equation}\label{4eleven}
       \abs{I}
          2^{-k(1-s)}\left\|\prod_{j=1}^{3}(F_{j}f_{j})(\alpha_{j},\vbeta)\right\|_{L^{\infty}_{\alpha,\mu (\vbeta)}} \int_{W_{1}\setminus W_{2^{k}}}       
           d\mu(\vbeta).
\end{equation}
By definition \eqref{defwgt} and direct calculation via polar coordinates $r=|\vbeta-\vgamma|$, we obtain
\begin{equation}
    \int_{W_{1}\setminus W_{2^{k}}}
    \frac{d\mathcal{H}^{2}(\vbeta)}{d_{\Gamma}(\vbeta)^2}
     \lesssim \int_{W_{1}\setminus W_{2^{k}}}
    \frac{d\mathcal{H}^{2}(\vbeta)}{|\vbeta-\vgamma|^2}
      \lesssim  
    \int_{1/|I|}^{2^{k}/|I|}\frac{1}{r^{2}}\cdot rdr \lesssim k .
\end{equation}
This together with the definition of $S_j$ for 
\eqref{4eleven} gives the bound 
\begin{equation}\label{eq414}
        \leq \abs{I}k\cdot 2^{k\br{1-s}}\cdot \prod_{j=1}^{3}\|F_{j}\|_{S^{j}},
\end{equation}
which completes our estimation of the first piece of \eqref{foureleven}.

Next, to calculate the second piece, we rewrite the remaining piece of \eqref{foureleven} into $L^{2}$ average 
\begin{equation}\label{eq417}
      2^{k(1-s)}\int_{W_{2^{k}}}
            \int_{I}
           (F_{i}f_{i})(\alpha_{i},\vbeta)
      \cdot 
        \prod_{j\neq i}\left\|(F_{j}f_{j})(\alpha_{j},\vbeta)  \right\|_{\L^{2}_{\alpha_{j}}\left(Q_{\alpha_{i}}\right)}
        d\alpha_i d\mathcal\mu(\vbeta).
\end{equation}
We decompose further the domain of integration into
\begin{equation}\label{inclu415}
    W_{2^{k}}\subseteq V_i\cup V_{i'}\cup V_{i''}
\end{equation}
where
\begin{equation}
  V_{a}:=\left( W_{2^{k}}\setminus \bigsqcup_{b\neq a} U^{b}_{\vgamma}\right).
\end{equation}
The inclusion \eqref{inclu415} holds because $U_{\vgamma}^{a}\cap U_{\vgamma}^{b}=\emptyset$ for $a\neq b$.
By the triangle inequality,
it suffices to estimate each piece independently. First, we deal with the piece $V_{i'}$. The the treatment of $V_{i''}$ is similar.
    \begin{equation}\label{tentestthree}
         |I|\cdot 2^{k(1-s)}\int_{V_{i'}}
            \frac{1}{|I|}\int_{I}
           (F_{i}f_{i})(\alpha_{i},\vbeta)
        \cdot 
     \prod_{j\neq i}\left\|(F_{j}f_{j})(\alpha_{j},\vbeta)  \right\|_{\L^{2}_{\alpha_{j}}\left(Q_{\alpha_{i}}\right)}
        d\alpha_i d\mathcal\mu(\vbeta)
    \end{equation}
 We perform a H\"older's inequality over the measure $d\alpha_{i}d\mu (\vbeta)$ where we place a $L^{\infty}$ norm on $i'$-th coordinate and $L^{2}$ norm on the rest. Then \eqref{tentestthree} is dominated by

 \begin{equation*}
     |I|\cdot 2^{k(1-s)}\left(\frac{1}{|I|}\int_{V_{i'}}\int_{I} |(F_{i}f_{i})(\alpha_{i},\vbeta)|^{2}d\alpha_{i}d\mu (\vbeta)  \right)^{\frac{1}{2}}
 \end{equation*}
 \begin{equation*}
     \cdot \left\| \left\| (F_{i'}f_{i'})(\alpha_{i'},\vbeta) \right\|_{\L^{2}_{\alpha_{i'}}\left(Q_{\alpha_{i}}\right)} \right\|_{L^{\infty}_{\alpha_{i},\mu (\vbeta)}}
 \end{equation*}
 \begin{equation*}
     \cdot \left(\frac{1}{|I|}\int_{V_{i'}}\int_{I} \left\| (F_{i''}f_{i''})(\alpha_{i''},\vbeta) \right\|^{2}_{\L^{2}_{\alpha_{i''}}\left(Q_{\alpha_{i}}\right)}d\alpha_{i}d\mu (\vbeta)  \right)^{\frac{1}{2}}.
 \end{equation*}

 Then by definition of size, we can further estimate
\begin{equation}\label{eqtent421}
        \lesssim  2^{k(1-s)} |I| \|F_{i}f_{i}\|_{S^{i}}
         \cdot \prod_{j\neq i}\left\|
      \left\|(F_{j}f_{j})(\alpha_{j},\vbeta)  \right\|_{\L^{2}_{\alpha_{j}}\left(Q_{\alpha_{i}}\right)} \right\|_{S^{j}(I,\vgamma,\frac{2^{k}}{|I|})}.
\end{equation}

Next, we deal with the piece $V_{i}$. Again, We perform a H\"older's inequality over the measure $d\alpha_{i}d\mu (\vbeta)$ where we place a $L^{\infty}$ norm on $i$-th coordinate and $L^{2}$ norm on the rest. Then \eqref{tentestthree} with $V_{i'}$ replaced by $V_{i}$ is dominated by
\begin{equation*}
    |I|\cdot 2^{k(1-s)}\left\| \br{F_{i}f_{i}}(\alpha_{i},\vbeta) \right\|_{L^{\infty}_{\alpha_{i},\mu (\vbeta)}}   
\end{equation*}
\begin{equation*}
    \cdot  \left(\frac{1}{|I|}\int_{V_{i}} \int_{I}\left\| (F_{i'}f_{i'})(\alpha_{i'},\vbeta) \right\|^{2}_{\L^{2}_{\alpha_{i'}}\left(Q_{\alpha_{i}}\right)}d\alpha_{i}d\mu (\vbeta)  \right)^{\frac{1}{2}}.
\end{equation*}
\begin{equation*}
     \cdot \left(\frac{1}{|I|}\int_{V_{i}}\int_{I}\left\| (F_{i''}f_{i''})(\alpha_{i''},\vbeta) \right\|^{2}_{\L^{2}_{\alpha_{i''}}\left(Q_{\alpha_{i}}\right)}d\alpha_{i}d\mu (\vbeta)  \right)^{\frac{1}{2}}.
\end{equation*}
Again, by the definition of size, the above term can be dominated by \eqref{eqtent421}.
We may rewrite each factor in the product of \eqref{eqtent421} into
\begin{equation}\label{eqtent422}
            \left\| 
                \left\| (F_{j}f_{j})\br{\alpha+x\frac{2^{k+1}}{d_{\Gamma}(\vbeta)},\:\vbeta}
                \right\|_{\L^{2}_{x}([-1,1])} 
            \right\|_{S^{j}(I,\vgamma,\frac{2^{k}}{|I|})}.
        \end{equation}
Since size is the maximum of a $L^{2}$ quantity and a $L^{\infty}$ quantity, by Minkowski inequality, we can estimate \eqref{eqtent422} by 
\begin{equation}\label{eq421}
            \left\| 
                \left\| 
              (F_{j}f_{j})\br{\alpha +x \frac{2^{k+1}}{d_{\Gamma}(\vbeta)},\:\vbeta}
                \right\|_{S^{j}(I,\vgamma,\frac{2^{k}}{|I|})}
            \right\|_{\L^{2}_{x}([-1,1])} 
\end{equation}
Fix $x$ and consider the inner norm in \eqref{eq421}. Observe that
\begin{equation*}
T:\br{\alpha,\vbeta}\mapsto\br{\alpha+x\frac{2^{k+1}}{d_\Gamma\br{\vbeta}},\:\vbeta}
\end{equation*}
is a map that preserves measure and for $\vbeta \in W_{2^{k}}$
\begin{equation*}
     \abs{x \frac{2^{k+1}}{d_\Gamma\br{\vbeta}}}\leq  \frac{2^{k+1}}{\delta_{1}|\vbeta-\vgamma|}\leq \frac{2^{k+1}}{\delta_{1}2^{k}/|I|}\leq 2\delta_{1}^{-1}\abs{I}.
\end{equation*}
Thus we have for \(C= \br{1+4\delta^{-1}_1}\),
\begin{equation*}
   T\br{I\times W_{2^{k}} }\subseteq CI\times W_{2^{k}}  \subseteq CI\times W_{\frac{1}{C}}. 
\end{equation*}
Then \eqref{eqtent421} can be estimated by
\begin{equation}\label{eq424}
        \lesssim  \abs{I}\cdot 2^{k\br{1-s}}
            \|(F_{i}f_{i})\|_{S^{i}}
        \cdot 
        \prod_{j\neq i}
            \left\| 
                \left\| (F_{j}f_{j})(\alpha ,\vbeta)
                \right\|_{S^{j}\br{CI \times W_{\vgamma,2^{k}/|CI|}}}
            \right\|_{\L^{2}_{x}([-1,1])}. 
\end{equation}
Then by the definition of size \eqref{sizedef}, 
\begin{equation*}
         \eqref{eq424}\lesssim  \abs{I}\cdot 2^{k\br{1-s}}\prod_{j=1}^3
            \nrm{
              F_{j}f_{j}
            }_{S^j}.
\end{equation*}
Together with \eqref{eq414}, we complete the proof of Proposition \ref{propTentest}.

\section{Proof of Proposition \ref{prop_Linfty_emb}: Global Estimate}
For \(\br{\alpha,\vbeta}\in\R\times V\), we introduce the \(L^1\) normalized wave-packet 
\begin{equation}
    \varphi_{\alpha,\vbeta}^{j}:=
        \Tr_{\alpha}\Mod_{-\beta_{j}}\Dil^{1}_{d_{\Gamma}(\vbeta)^{-1}}\overline{\varphi}.
\end{equation}
Since \(\varphi\) is even, direct calculation gives
\begin{equation*}
(F_{j}f)(\alpha,\vbeta)=\left((\Mod_{\beta_{j}}\Dil_{d_{\Gamma}(\vbeta)^{-1}}^{1}\varphi) \ast f_{j}\right)(\alpha)=\left\langle f_{j},\varphi_{\alpha,\vbeta}^{j} \right\rangle.
\end{equation*}
By the definition of the global size, it suffices to show
\begin{equation*}
    \nrm{F_j f}_{S^j\br{I,\vgamma}}\lesssim \nrm{f}_{L^\infty},\quad \forall I\in\mathcal{I},\vgamma\in \Gamma.
\end{equation*}
Fix a pair \(\br{I,\vgamma}\in \mathcal{I}\times\Gamma\) and recall 
\begin{equation}\label{eq53}
    \nrm{F_j f}_{S^j\br{I,\vgamma}}:=
    \nrm{F_j f}_{L^\infty\br{I\times W_{\vgamma,1/\abs{I}}}}
    \vee
    \abs{I}^{-\frac{1}{2}}\nrm{F_jf}_{L^2_{\nu}\br{I\times \br{W_{\vgamma,1/\abs{I}}\setminus U_\vgamma^j}}}.
\end{equation}
Trivially, we have
\begin{equation}
    \abs{F_j f\br{\alpha,\vbeta}}=
    \abs{
        \ang{
            f,\varphi_{\alpha,\vbeta}^{j}
        }
    }\leq \nrm{f}_{L^\infty}.
    \label{eq_global_Linfty_holder}
\end{equation}
Hence, the \(L^\infty\) component in \eqref{eq53} is dominated by \(\nrm{f}_{L^\infty}\).

To control the remaining \(L^2\) component, we split the function \(f\) into the local part \(1_{3I}f\) and the tail part \(1_{3I^c}f\). By linearity of \(F_j\) and the triangle inequality, it suffices to obtain the following estimate for the local part
\begin{equation}
    \abs{I}^{-\frac{1}{2}}\nrm{F_j(1_{3I}f)}_{L^2_{\nu}\br{I\times \br{W_{\vgamma,1/\abs{I}}\setminus U_\vgamma^j}}}\lesssim \nrm{f}_{L^\infty}
    \label{eq_global_L2_loc}
\end{equation}
and the analogous estimate for the tail part
\begin{equation}
    \abs{I}^{-\frac{1}{2}}\nrm{F_j (1_{3I^c}f)}_{L^2_{\nu}\br{I\times \br{W_{\vgamma,1\abs{I}}\setminus U_\vgamma^j}}}\lesssim \nrm{f}_{L^\infty}.
    \label{eq_global_L2_tail}
\end{equation}

Starting with the treatment of the local part \eqref{eq_global_L2_loc}, we observe that it's enough to show the following \(L^2\) estimate
\begin{equation}
    \nrm{F_j g}_{L^2_\nu\br{\R\times \br{W_{\vgamma,0}\setminus U_\vgamma^j}}}\lesssim \nrm{g}_{L^2}.
    \label{eq_global_L2_g_est}
\end{equation}
This can be seen by taking \(g=1_{3I}f\),
\begin{equation*}
    \abs{I}^{-\frac{1}{2}}
    \nrm{
        F_j (1_{3I} f)
    }_{L^2_\nu\br{I\times \br{
        W_{\vgamma,1/\abs{I}}\setminus U_\vgamma^j
    }}}
    \leq
    \abs{I}^{-\frac{1}{2}}
    \nrm{F_j g}_{L^2_\nu\br{\R\times \br{W_{\vgamma,0}\setminus U_\vgamma^j}}}\lesssim \abs{I}^{-\frac{1}{2}}\nrm{g}_{L^2}.
\end{equation*}
The localization gives the trivial bound \(\abs{I}^{-\frac{1}{2}}\nrm{g}_{L^2}\lesssim \nrm{f}_{L^\infty}\), which finishes \eqref{eq_global_L2_loc}. Return to the proof of \eqref{eq_global_L2_g_est}. Recall that
\begin{equation*}
    F_j g\br{\alpha,\vbeta}:=
    \br{
        M_{\beta_j}
        D^1_{d_\Gamma\br{\vbeta}^{-1}}
        \varphi
    }
    \ast g\br{\alpha}.
\end{equation*}
Applying Plancherel on the spatial variable \(\alpha\) to the left-hand side of yields \eqref{eq_global_L2_g_est}
\begin{equation}\label{eq510}
   \nrm{
        \nrm{
            \br{
                T_{\beta_j} D^\infty_{d_\Gamma\br{\vbeta}}
                \widehat{\varphi}
            }\br{\xi}
            \cdot \widehat{g}\br{\xi}
        }_{L^2_{\xi}}
    }_{L^2_{\mu\br{\vbeta}}\br{W_{\vgamma,0}\setminus U_\vgamma^j}}
\end{equation}
then interchange the order of the \(L^2\) norms equates \eqref{eq510} to
\begin{equation}
    \nrm{
        \nrm{
            \br{
                T_{\beta_j} D^\infty_{d_\Gamma\br{\vbeta}}
                \widehat{\varphi}
            }\br{\xi}
        }_{L^2_{\mu\br{\vbeta}}\br{W_{\vgamma,0}\setminus U_\vgamma^j}}
        \cdot \widehat{g}\br{\xi}
    }_{L^2_{\xi}}.
\end{equation}
It remains to show 
\begin{equation}\label{eq512}
    \nrm{
        T_{\beta_j} D^\infty_{d_\Gamma\br{\vbeta}}
        \widehat{\varphi}
    }_{L^2_{\mu(\vbeta)}\br{ W_{\vgamma,0}\setminus U_\vgamma^j}}\lesssim 1.
\end{equation}        
By developing the \(L^2\) norm, the left-hand side of \eqref{eq512} equals to
\begin{equation}
    \int_{W_{\vgamma,0}\setminus U_\vgamma^j}
        \abs{
            \widehat{\varphi}
        }^2\br{\frac{\xi-\beta_j}{d_\Gamma\br{\vbeta}}}
    \frac{d\mathcal{H}^2\br{\vbeta}}{d_\Gamma\br{\vbeta}^2}.
    \label{eq_global_L2_mul_est}
\end{equation}
Recall that for \(\vbeta\in W_{\vgamma,0}\setminus U_\vgamma^j\),
\begin{equation*}
    \delta_1\abs{\beta_j-\gamma_j}\leq
    \delta_1\abs{\vbeta-\vgamma}\leq 
    d_\Gamma\br{\vbeta}\leq 
    \abs{\vbeta-\vgamma}\leq
    \frac{1}{\delta_2}\abs{\beta_j-\gamma_j}.
\end{equation*}
On the other hand, since \(\supp \widehat{\varphi}\subset B_{\frac{4\epsilon}{10}}\br{0}\),
\begin{equation*}
    \abs{
        \widehat{\varphi}
    }^2\br{\frac{\xi-\beta_j}{d_\Gamma\br{\vbeta}}}
    \neq 0
    \implies
    \abs{\xi-\beta_j}\leq \frac{4\epsilon}{10}d_\Gamma\br{\vbeta}\leq \frac{2\epsilon}{5\delta_2}\abs{\beta_j-\gamma_j}.
\end{equation*}
By triangle inequality, 
\begin{equation*}
    \abs{\beta_j-\gamma_j}-\abs{\xi-\beta_j}\leq\abs{\xi-\gamma_j}\leq\abs{\beta_j-\gamma_j}+\abs{\xi-\beta_j}.
\end{equation*}
Hence,
\begin{equation*}
    \br{1-\frac{2\epsilon}{5\delta_2}}\abs{\beta_j-\gamma_j}\leq
    \abs{\xi-\gamma_j}\leq
    \br{1+\frac{2\epsilon}{5\delta_2}}\abs{\beta_j-\gamma_j}.
\end{equation*}
This shows that \(
    \abs{\vbeta-\vgamma}\sim d_\Gamma\br{\vbeta}\sim \abs{\xi-\gamma_j}
\). As a direct consequence, we can dominate \eqref{eq_global_L2_mul_est} by
\begin{equation*}
    \nrm{\widehat{\varphi}}_{L^\infty}^2
    \cdot
    \fint_{
        \left\{
            \vbeta\in V 
        \middle\vert
            \abs{\vbeta-\vgamma}\sim \abs{\xi-\gamma_j}
        \right\}
    }
    d\mathcal{H}^2\br{\vbeta}\sim 1
\end{equation*}
and thus, verify \eqref{eq512} and complete the proof of \eqref{eq_global_L2_g_est}.

As for the tail part \eqref{eq_global_L2_tail}, we have a trivial bound
\begin{equation*}
    \abs{I}^{-\frac{1}{2}}\nrm{F_j(1_{3I^c}f)}_{L^2_{\nu}\br{I\times \br{W_{\vgamma,1/\abs{I}}\setminus U_\vgamma^j}}}
\end{equation*}
\begin{equation}\label{eq_IFL1Linfty}
    \leq
    \abs{I}^{-\frac{1}{2}}\nrm{F_j(1_{3I^c}f)}_{L^1_{\nu}\br{I\times \br{W_{\vgamma,1/\abs{I}}\setminus U_\vgamma^j}}}^{\frac{1}{2}}\cdot \nrm{F_j (1_{3I^c}f)}_{L^\infty}^{\frac{1}{2}}.
\end{equation}
By \eqref{eq_global_Linfty_holder}, we dominate \eqref{eq_IFL1Linfty} by
\begin{equation*}
    \abs{I}^{-\frac{1}{2}}\nrm{F_j(1_{3I^c}f)}_{L^1_{\nu}\br{I\times \br{W_{\vgamma,1/\abs{I}}\setminus U_\vgamma^j}}}^{\frac{1}{2}}\cdot \nrm{f}_{L^\infty}^{\frac{1}{2}}.
\end{equation*}
Therefore, to show \eqref{eq_global_L2_tail}, it suffices to show the following inequality
\begin{equation}\label{eq_IFlessf}
\abs{I}^{-1}\nrm{F_j(1_{3I^c}f)}_{L^1_{\nu}\br{I\times W_{\vgamma,1/\abs{I}}}}\lesssim \nrm{f}_{L^\infty}.
\end{equation}
We may dominate the left-hand side of \eqref{eq_IFlessf} by
\begin{equation*}
    \nrm{f}_{L^\infty}\!\!
    \int_{3I^c}
        \fint_I
            \int_{W_{\vgamma,1/\abs{I}}}
                \abs{       
                    \varphi_{\alpha,\vbeta}^j
                }\br{x}
            d\mu\br{\vbeta}
        d\alpha
    dx.
\end{equation*}
Then, it suffices to show
\begin{equation}\label{eq_global_L2_tail_abs_IIW}
    \int_{3I^c}
        \fint_I
            \int_{W_{\vgamma,1/\abs{I}}}
                \abs{       
                    \varphi_{\alpha,\vbeta}^j
                }\br{x}
            d\mu\br{\vbeta}
        d\alpha
    dx\lesssim 1.
\end{equation}
Recall again that \(\vbeta\in W_{\vgamma,1/\abs{I}}\) implies that \(d_\Gamma\br{\vbeta}\sim \abs{\vbeta-\vgamma}\) and thus,
\begin{equation*}
    \abs{\varphi_{\alpha,\vbeta}^j}\br{x}\lesssim
    \abs{\vbeta-\vgamma}\br{1+\abs{\vbeta-\vgamma}\cdot\abs{x-\alpha}}^{-N}.
\end{equation*}
By polar coordinates with center at \(\vgamma\), we have 
\begin{equation*}
    \int_{W_{\vgamma,1/\abs{I}}}
        \abs{       
            \varphi_{\alpha,\vbeta}^j
        }\br{x}
    d\mu\br{\vbeta}
    \lesssim
    \int^\infty_{1/\abs{I}}
        t\br{1+t\abs{x-\alpha}}^{-N}
    \frac{dt}{t}.
\end{equation*}
As a result, the left-hand side of \eqref{eq_global_L2_tail_abs_IIW} is dominated by
\begin{equation}
    \lesssim
    \int_{3I^c}
        \fint_I
            \int^\infty_{1/\abs{I}}
                t\br{1+t\abs{x-\alpha}}^{-N}
            \frac{dt}{t}
        d\alpha
    dx.
    \label{eq_global_L2_tail_abs_III}
\end{equation}
After a change of variable,
\begin{equation*}
    \eqref{eq_global_L2_tail_abs_III}=\int_{\Br{-\frac{3}{2},\frac{3}{2}}}
        \int_{\Br{-\frac{1}{2},\frac{1}{2}}}
            \int^\infty_1
                t\br{1+t\abs{x-\alpha}}^{-N}
            \frac{dt}{t}
        d\alpha
    dx.
\end{equation*}
It becomes apparent that the above quantity can be estimated by
\begin{equation*}
    \sim
    \int_{\Br{-\frac{3}{2},\frac{3}{2}}}
        \int_{\Br{-\frac{1}{2},\frac{1}{2}}}
            \int^\infty_1
                t^{1-N}
            \frac{dt}{t}
            \cdot
            \abs{x-\alpha}^{-N}
        d\alpha
    dx
    =
    \int_{\Br{-\frac{3}{2},\frac{3}{2}}}
        \int_{\Br{-\frac{1}{2},\frac{1}{2}}}
            \abs{x-\alpha}^{-N}
        d\alpha
    dx
    \sim 1.
\end{equation*}
This verifies \eqref{eq_global_L2_tail_abs_IIW} and completes the proof of \eqref{eq_global_L2_tail}.

\section{Proof of Proposition \ref{proptentgeo}: Geometry of Tents}
For $x_{1},x_{2} \in \mathbb{R}^{n}$, $0<r<1$, define the Apollonian circle
 \begin{equation}
        B_{r}(x_{1},x_{2}):=\left\{y \in \mathbb{R}^{n} : \frac{|y -x_{1}|}{r}<\frac{|y -x_{2}|}{1}  \right\}.
    \end{equation}
We begin with a geometric lemma concerning the relation between two Apollonian circles.
\begin{lemma}\label{appolem}
    Let $x_{0},x_{1},x_{2}\in \mathbb{R}^{n}$ and $0<r<1$. Suppose that 
    \begin{equation}\label{appoasump}
        r|x_{2}-x_{1}|\leq |x_{2}-x_{0}|-|x_{1}-x_{0}|,
    \end{equation}
    then $B_{r}(x_{0},x_{1})\subseteq B_{r}(x_{0},x_{2})$. This inclusion relation is equivalent to the fact that if $y\in \mathbb{R}^{n}$ satisfies
    \begin{equation*}
         r|y -x_{2}|\leq |y -x_{0}|,
    \end{equation*}
    it must also satisfy
    \begin{equation*}
         r|y -x_{1}|\leq |y -x_{0}|.
    \end{equation*}
\end{lemma}
\begin{proof}
    By direct calculation, we have the center $C(B_{r}(x_{0},x_{i}))$ and radius $R(B_{r}(x_{0},x_{i}))$ of these two Apollonian circles
    \begin{equation}
        C(B_{r}(x_{0},x_{i}))=\frac{x_{0}-r^{2}x_{i}}{1-r^{2}},\quad R(B_{r}(x_{0},x_{i}))=\frac{r|x_{0}-x_{i}|}{1-r^{2}}.
    \end{equation}
Then by assumption \eqref{appoasump},
\begin{equation*}
    R(B_{r}(x_{0},x_{2}))-R(B_{r}(x_{0},x_{1}))=\frac{r(|x_{2}-x_{0}|-|x_{1}-x_{0}|)}{1-r^{2}}
\end{equation*}
\begin{equation*}
     \geq \frac{r^{2}|x_{2}-x_{1}|}{1-r^{2}}
    =| C(B_{r}(x_{0},x_{2}))-C(B_{r}(x_{0},x_{1}))|.
\end{equation*}
Hence $B_{r}(x_{0},x_{1})\subseteq B_{r}(x_{0},x_{2})$.
\end{proof}

\begin{figure}[h]
    \centering
\includegraphics[width=0.4\textwidth,height=4.5cm]{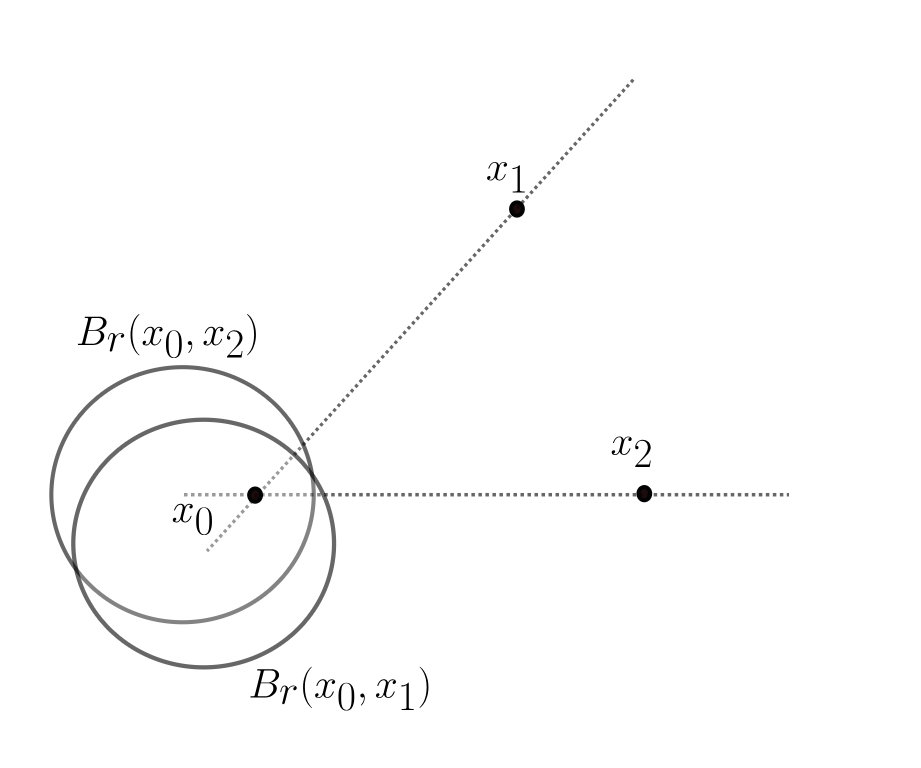}
    \caption{The inclusion relation between two Apollonian circles.}
    \label{fig:enter-label}
\end{figure}

\begin{lemma}\label{UinW}
    For $\vgamma \in \Gamma$ and $j=1,2,3$, we have the inclusion relation
    \begin{equation}
        U_{\vgamma}^{j}\subseteq W_{\vgamma , 0}.
    \end{equation}
\end{lemma}
\begin{proof}
Notice that
\begin{equation}
    W_{\vgamma,0}=V\setminus \bigcup_{\vgamma' \in \Gamma}B_{\delta_{1}}(\vgamma' ,\vgamma).
\end{equation}
For $\vbeta \in  \underset{\vgamma' \in \Gamma}{\bigcup}B_{\delta_{1}}(\vgamma' ,\vgamma)$, there is $\vgamma' \in \Gamma$ such that $\vbeta \in B_{\delta_{1}}(\vgamma' ,\vgamma)$. Recall that $B_{\delta_{1}}(\vgamma' ,\vgamma)$ is a ball with center
\[
C(B_{\delta_{1}}(\vgamma' ,\vgamma))=\frac{\vgamma'-\delta_{1}^{2}\vgamma}{1-\delta_{1}^{2}}=\vgamma +\frac{\vgamma'-\vgamma}{1-\delta_{1}^{2}}
\]
and radius
\[
R(B_{\delta_{1}}(\vgamma' ,\vgamma))=\delta_{1}\cdot \frac{\vgamma'-\vgamma}{1-\delta_{1}^{2}}.
\]
The angle between $\overline{\vgamma \: \vbeta}$ and $\overline{\vgamma \:C(B_{\delta_{1}}(\vgamma' ,\vgamma))}$ is at most 
\[
\operatorname{arcsin}\left( \frac{C(B_{\delta_{1}}(\vgamma' ,\vgamma))-\vgamma}{R(B_{\delta_{1}}(\vgamma' ,\vgamma))}\right)=\operatorname{arcsin}\delta_{1}=\theta_{1}.
\]
Since $\vgamma, \vgamma' \in \Gamma$, by Lemma \ref{lemawayfromortho}, the angle between $\overline{\vgamma \:\vgamma'}$ and $P_{V}e_{j}$ is at most $\frac{\pi}{3}+\theta_{0}$. 
Using the fact that $\vgamma ,\vgamma', C(B_{\delta_{1}}(\vgamma' ,\vgamma))$ are on the same line, we conclude that the angle between \(\overline{\vgamma\vbeta}\) and \(P_{V}e_{j}\) is at most \(\frac{\pi}{3}+\theta_0+\theta_1\). Finally, we recall the definition \eqref{Uangle} of $U_{\gamma}^{j}$ and obtain $U_{\vgamma}^{j}\subseteq W_{\vgamma ,0}$.
\end{proof}
\begin{lemma}\label{pres_order_appo}
    For $\vgamma, \vgamma', \vgamma'' \in \Gamma$. Suppose we have an order relation $\vgamma_{j}\leq {\vgamma'}_{j} \leq {\vgamma''}_{j}$ in some direction $j\in \{1,2,3\}$, then the order is either preserved $\vgamma_{i}\leq {\vgamma'}_{i} \leq {\vgamma''}_{i}$ or  reversed ${\vgamma''}_{i} \leq {\vgamma'}_{i} \leq {\vgamma}_{i}$ in other directions $i\neq j$, $i\in \{1,2,3\}$. Furthermore, $\vgamma, \vgamma', \vgamma''$ satisfy the condition \eqref{appoasump},
\begin{equation}
    \delta_{1}|\vgamma''-\vgamma'|\leq |\vgamma''-\vgamma|- |\vgamma'-\vgamma|.
\end{equation}
\end{lemma}
\begin{proof}
    Translate $\vgamma'$ to origin. By condition \eqref{stay_in_one_cond}, we have $\vgamma$ and $\vgamma''$ in two octants diagonal to each other. This implies that in \(i\)-th direction, the order is either preserved or reversed. Notice that by \eqref{stay_in_one_cond}, we have 
    \[
   \langle \vgamma-\vgamma',\vgamma''-\vgamma'\rangle \leq -\frac{1}{2}|\vgamma-\vgamma'|\cdot |\vgamma''-\vgamma'|.
    \]
Then
\begin{equation*}
    |\vgamma''-\vgamma|^{2}=|\vgamma'-\vgamma|^{2}+|\vgamma''-\vgamma'|^{2}-2\langle \vgamma-\vgamma',\vgamma''-\vgamma'\rangle 
\end{equation*}
\begin{equation*}
     \geq |\vgamma'-\vgamma|^{2}+|\vgamma''-\vgamma'|^{2}+|\vgamma'-\vgamma|\cdot |\vgamma''-\vgamma'|
\end{equation*}
\begin{equation*}
     \geq \left(|\vgamma'-\vgamma|+\frac{1}{2} |\vgamma''-\vgamma'|\right)^{2}.
\end{equation*}
Take square root on both sides and by $\delta_{1}\leq \frac{1}{2}$, we have the desired result.
\end{proof}
Next, we prove the first statement in Proposition \ref{proptentgeo}. 
\begin{proof}[Proof of Proposition \ref{proptentgeo} (1). ]
For $\vbeta \in V$, since $\Gamma$ is closed, there exists a point  $\vgamma(\vbeta)=\br{\gamma(\vbeta)_{1},\gamma(\vbeta)_{2},\gamma(\vbeta)_{3}}$ on the singularity $\Gamma$ such that $|\vbeta-\vgamma(\vbeta)|=d_{\Gamma}(\vbeta)$. Suppose $\vbeta \in W_{\vgamma'' ,t}$. First, we show that $\vbeta \notin B(\vgamma,\delta_1 t)$. The argument for $\vbeta\notin B(\vgamma',\delta_1 t)$ is the same. By \eqref{singcond} and \eqref{meshcond},
\begin{equation}
    |\vgamma''-\vgamma|\leq \frac{1}{\delta_{0}}|{\gamma''}_{j}-\gamma_{j}|\leq (1-\delta_1)t.\label{eq_gg'g''inbetween}
\end{equation}
Then we have
\begin{equation*}
    |\vbeta -\vgamma|\geq |\vbeta -\vgamma''|-|\vgamma''-\vgamma|\geq t-(1-\delta_1)t=\delta_1 t.
\end{equation*}
Second, we show either $\delta_{1}|\vbeta -\vgamma|\leq d_{\Gamma}(\vbeta)$ or  $\delta_{1}|\vbeta -\vgamma'|\leq d_{\Gamma}(\vbeta)$. According to the location of $\gamma (\vbeta)_{j}$, we divide into three cases: $(1) :\gamma_{j}\leq \gamma (\vbeta)_{j}\leq {\gamma'}_{j}$, $(2) :\gamma (\vbeta)_{j}<\gamma_{j}$, $(3) :{\gamma'}_{j}<\gamma (\vbeta)_{j}$.
For case \((1)\), via an augment similar to \eqref{eq_gg'g''inbetween} with \(\vgamma''\) replaced by \(\vgamma\br{\vbeta}\) and the assumption $\vbeta \in W_{\vgamma'',t}$, we obtain
\begin{equation*}
    |\vbeta -\gamma|\leq |\vbeta- \vgamma (\vbeta)|+|\vgamma (\vbeta)-\vgamma|\leq d_{\Gamma}(\vbeta)+(1-\delta_1)t\leq \frac{1}{\delta_1} d_{\Gamma}(\vbeta).
\end{equation*}
% Then by \eqref{rhorange}, we have $1+\frac{1-\delta_1}{ \delta_{1}}\leq \frac{1}{\delta_{1}}$. 

Case \((2)\) and \((3)\) are symmetric. Hence, we only prove the case \((2)\). In this case, $\vgamma (\vbeta), \vgamma, \vgamma''$ satisfy the relation
\begin{equation*}
   \gamma (\vbeta)_{j}\leq \gamma_{j}\leq {\gamma''}_{j}.
\end{equation*}
By Lemma \ref{pres_order_appo},
\begin{equation*}
    \delta_{1} |\vgamma'' -\vgamma|\leq |\vgamma'' -\vgamma (\vbeta)|-|\vgamma -\vgamma (\vbeta)|.
\end{equation*}
Then by Lemma \ref{appolem}, 
\begin{equation*}
    \br{
        \: \delta_{1}|\vbeta -\vgamma''| \leq |\vbeta -\vgamma (\vbeta)| \:
    }
    \implies
    \br{
        \: \delta_{1}|\vbeta -\vgamma| \leq |\vbeta -\vgamma (\vbeta)|  \:
    }.
\end{equation*}
Since $\vbeta \in W_{\gamma'' ,t}$,
\begin{equation*}
    \delta_{1}|\vbeta -\vgamma''|\leq d_{\Gamma}(\vbeta)=|\vbeta -\vgamma (\vbeta)|.
\end{equation*}
We then have
\begin{equation*}
       \delta_{1}|\vbeta -\vgamma|\leq |\vbeta -\vgamma (\vbeta)|=d_{\Gamma}(\vbeta)
\end{equation*}
and complete the proof of statement \textit{(1)} in Proposition \ref{proptentgeo}.
\end{proof}

Before proving statement \textit{(2)} in Proposition \ref{proptentgeo}, we introduce a Lemma.
\begin{lemma}\label{rhslemma}
    We have the inclusion
    \begin{equation}
        W_{\vgamma',0}\setminus W_{\vgamma,t}\subseteq B_{t}(\vgamma)\cup \left( W_{\vgamma' ,0}\setminus U_{\vgamma}^{j}\right)^{>\gamma_{j}},
    \end{equation}
    where $\left( W_{\vgamma' ,0}\setminus U_{\vgamma}^{j}\right)^{>\gamma_{j}}$ consists of all the points $\vbeta$ in $W_{\vgamma' ,0}\setminus U_{\vgamma}^{j}$ with $\beta_{j}>\gamma_{j}$.
\end{lemma}
\begin{proof}
    Let $\vbeta\in  W_{\vgamma',0}\setminus W_{\vgamma,t}$, we may split into four cases.\\
    
    Case 1:  $\gamma(\vbeta)_{j}\leq \gamma_{j}$.\\
    Since $\gamma(\vbeta)_{j}\leq \gamma_{j}\leq \gamma'_{j}$, by Lemma \ref{appolem} and Lemma \ref{pres_order_appo}, 
\begin{equation}
    \br{
        \: \delta_{1}|\vbeta -\vgamma'| \leq |\vbeta -\vgamma (\vbeta)| \:
    }
    \implies
    \br{
        \: \delta_{1}|\vbeta -\vgamma| \leq |\vbeta -\vgamma (\vbeta)|  \:
    }.
\end{equation}
Since $\vbeta \in W_{\vgamma',0}$, we then have $\vbeta \in W_{\vgamma ,0}$. By assumption, $\vbeta \notin W_{\vgamma,t}$. Hence $\vbeta\in B_{t}(\vgamma)$.\\

Case 2: $\vbeta \in U_{\vgamma}^{j}$.\\
By Lemma \ref{UinW}, $\vbeta \in W_{\vgamma ,0}$. By assumption, $\vbeta \notin W_{\vgamma,t}$. Hence $\vbeta\in B_{t}(\vgamma)$.\\

Case 3: $\gamma_{j}<\beta_{j}$.\\
Since, $\vbeta\in  W_{\vgamma',0}\setminus W_{\vgamma,t}\subseteq W_{\vgamma' ,0}\setminus U_{\vgamma}^{j}$, we have $\vbeta\in   \left(W_{\vgamma' ,0}\setminus U_{\vgamma}^{j}\right)^{>\gamma_{j}}$.\\

Case 4: $\vbeta \notin U_{\vgamma}^{j}$ and $\beta_{j}<\gamma_{j}<\gamma(\vbeta)_{j}$.\\
 Notice that
\begin{equation}
    W_{\vgamma,0}=V\setminus \bigcup_{\widetilde{\vgamma} \in \Gamma}B_{\delta_{1}}(\widetilde{\vgamma} ,\vgamma).
\end{equation}
Now that $V\setminus U_{\vgamma}^{j}$ has two connected components. On the one hand, $\gamma_{j}<\gamma(\vbeta)_{j}$ and thus, $B_{\delta_{1}}(\vgamma(\vbeta) ,\vgamma)$ lies in the right component. On the other hand, $\beta_{j}<\gamma_{j}$ and thus, $\vbeta$ lies in the left component. Hence $\vbeta \notin B_{\delta_{1}}(\vgamma(\vbeta) ,\vgamma)$. Unpacking the definition of $B_{\delta_{1}}(\vgamma(\vbeta) ,\vgamma)$ and $W_{\vgamma,0}$, we obtain $\vbeta \in W_{\vgamma ,0}$. Together with $\vbeta \notin W_{\vgamma,t}$, we conclude $\vbeta\in B_{t}(\vgamma)$.
\end{proof}
We finish this section with the proof of statement \textit{(2)} in Proposition \ref{proptentgeo}. 
\begin{proof}[Proof of Proposition \ref{proptentgeo} (2).]
By Lemma \ref{rhslemma}, 
\begin{equation}
        W_{\vgamma',0}\setminus W_{\vgamma,\delta_1 t}\subseteq B_{\delta_1 t}(\vgamma)\cup \left( W_{\vgamma' ,0}\setminus U_{\vgamma}^{j}\right)^{>\gamma_{j}}.
\end{equation}
Suppose $\vbeta' \in B_{\delta_1 t}(\vgamma)$, we have the following estimate
\begin{equation*}
    |{\beta'}_{j}-\gamma_{j}|\leq |\vbeta'-\vgamma|<\delta_1 t\leq \delta_1 |\vbeta -\vgamma|.
\end{equation*}
With the assumption that $\vbeta \in W_{\vgamma,t}\setminus U_{\vgamma}^{j}$, we obtain
\begin{equation}\label{eq522}
        |\beta_{j}-{\beta'}_{j}|\geq |\beta_{j}-\gamma_{j}|-|{\beta'}_{j}-\gamma_{j}|\geq (\delta_{2}-\delta_1)|\vbeta -\vgamma|.
\end{equation}
We then split the previous term and further estimate \eqref{eq522} by  
\begin{equation*}
        (\delta_{2}-\delta_1)\frac{1}{1+\delta_1}|\vbeta -\vgamma|+(\delta_{2}-\delta_1)\frac{\delta_1}{1+\delta_1}|\vbeta -\vgamma|
\end{equation*}
\begin{equation*}
       \geq \frac{\delta_{2}-\delta_1}{1+\delta_1}\left(|\vbeta-\vgamma|+|\vbeta'-\vgamma| \right)
    \geq  \rho \left(d_{\Gamma}(\vbeta)+d_{\Gamma}(\vbeta') \right).
\end{equation*}
Suppose  $\vbeta' \in \left( W_{\vgamma' ,0}\setminus U_{\vgamma}^{j}\right)^{>\gamma_{j}}$, then $\beta_{j}<\gamma_{j}<{\beta'}_{j}$. Hence
\begin{equation*}
    |{\beta'}_{j}-\beta_{j}|=|{\beta'}_{j}-\gamma_{j}|+|\gamma_{j}-\beta_{j}|\geq \delta_{2}(|\vbeta' -\vgamma|+|\vbeta -\vgamma|)\geq  \rho \left(d_{\Gamma}(\vbeta')+d_{\Gamma}(\vbeta) \right).
\end{equation*}
% \begin{equation}
%     \geq  \delta_{2}\left(d_{\Gamma}(\vbeta')+d_{\Gamma}(\vbeta) \right) \geq  \rho \left(d_{\Gamma}(\vbeta')+d_{\Gamma}(\vbeta) \right).
% \end{equation}
We complete the proof of Proposition 
\ref{proptentgeo}.
\end{proof}

\section{Proof of Proposition \ref{prop_select}: Selection Algorithm, $L^{\infty}$ Component}
 Let $t_{0}:=\operatorname{dist}(P_{V}\Omega ,\Gamma)=\underset{\vbeta \in P_{V}\Omega}{\operatorname{inf}}d_{\Gamma}(\vbeta)$. By compactness, there's a positive distance between \(P_{V}\Omega\) and \(\Gamma\) and therefore \(t_0>0\). For $\vbeta \in V$, $j\in \{1,2,3\}$, we write the projection map $P_{j}\vbeta:=P_{\R \cdot \ve_j}\vbeta=\beta_{j}$ for simplicity. Let $\T_{0}=\varnothing$. Suppose we have $\T_{k}$ for $0\leq k\leq k_{0}$ we define
\begin{equation*}
    \P_{k_{0}}:=\bigg(\Omega \cap \left\vert F_{j}f\right\vert^{-1}(\lambda,2\lambda]
    \cap \left\{(\alpha,\vbeta)\in \mathbb{R}\times V  \setminus \Gamma :
    2^{k}t_{0} \leq d_{\Gamma}(\vbeta) < 2^{k+1}t_{0} \right\}\bigg)
\end{equation*}
\begin{equation*}
   \setminus \bigg( \bigcup_{k=0}^{k_{0}}\bigcup_{T\in \T_{k}}D_{T}\bigg).
\end{equation*}
For next iteration, suppose we have $\P_{k}$ and $\T_{k}$ for $0\leq k\leq k_{0}$, we construct $\T_{k_{0}+1}$ through the following process.
For $(\alpha ,\vbeta)\in \R\times V$, $t>0$ define rectangles
\begin{equation}
    R_{\alpha,\vbeta ,t}:=\left(\alpha + \frac{c_{s}}{t}\Br{-\frac{1}{2},\frac{1}{2}}\right)\times \left( 
 \gamma(\vbeta)_{j}+c_{f}t\Br{-\frac{1}{2},\frac{1}{2}} \right).
\end{equation}
where $c_{s}$ and $c_{f}$ are two constants to be determined later.
Let $\P_{k_{0}}^{\ast}$ be a finite subset of $\P_{k_{0}}$ such that for distinct point $(\alpha,\vbeta), (\alpha' ,\vbeta')\in \P_{k_{0}}^{\ast}$
\begin{equation}\label{Rdijonit63}
 R_{\alpha,\vbeta ,2^{k_{0}}t_{0}}\cap  R_{\alpha',\vbeta' ,2^{k_{0}}t_{0}}=\varnothing.
\end{equation}
and maximal in the sense that for any $(\alpha,\vbeta)\in \P_{k_{0}}$, there exists a $(\alpha',\vbeta')\in \P_{k_{0}}^{\ast}$ such that
\begin{equation}
    R_{\alpha,\vbeta ,2^{k_{0}}t_{0}}\cap  R_{\alpha',\vbeta' ,2^{k_{0}}t_{0}}\neq \varnothing.
\end{equation}
The existence of such finite set $\P_{k_{0}}^{\ast}$ is guaranteed by the compactness of $\Omega$ and a greedy algorithm. 
Given $\vbeta \in V \setminus \Gamma$, and for $i\in \mathbb{Z},\:-M\leq i\leq M$ with $M$ being the least integer greater than $\frac{3c_{f}}{2\delta_{0}(1-\delta_1)}$, we define
\begin{equation}
    \Gamma^{i}_{\vbeta}:=\Gamma \cap P_{j}^{-1}\left(\gamma(\vbeta)_{j}+\delta_{0}(1-\delta_1) 2^{k_{0}}t_{0}[i-1,i]  \right).
\end{equation}
If $\Gamma^{i}_{\vbeta}\neq \varnothing$, by closedness of \(\Gamma^i\), there exists $\vgamma_{-}^{(i)}(\vbeta),\vgamma_{+}^{(i)}(\vbeta)\in \Gamma^{i}_{\vbeta}$ such that $P_{j}(\Gamma^{i}_{\vbeta})\subseteq [\gamma_{-}^{(i)}(\vbeta)_{j},\gamma_{+}^{(i)}(\vbeta)_{j}]$. Define 
\begin{equation}
    \T_{k_{0}+1}:=\bigcup_{i=-M}^{M}\left\{\left(\alpha+\frac{c}{2^{k_{0}}t_{0}}\Br{-\frac{1}{2},\frac{1}{2}},\vgamma_{\pm}^{(i)}({\vbeta})\right):(\alpha,\vbeta)\in P_{k_{0}}^{\ast}\right\},
\end{equation}
where $c=(3c_{s}\vee\frac{1}{\delta_1})$. We will show that $\P=\underset{k\geq 0}{\bigcup}\P_{k}^{\ast}$ and $\T=\underset{k\geq 0}{\bigcup}\T_{k}$ satisfy the desired properties in Proposition \ref{prop_select}. To show the covering property \eqref{select1-1}, we first recall that by construction \(\P_k^\ast\subset \Omega \cap \abs{F_j f}^{-1}\left(\lambda,2\lambda\right]\), and thus \(\P \subset \Omega \cap \abs{F_j f}^{-1}\left(\lambda,2\lambda\right] \). On the other hand, for \(\br{\alpha,\vbeta}\in \Omega \cap \abs{F_j f}^{-1}\left(\lambda,2\lambda\right]\), there is \(k_0\) such that \(2^{k_0}t_0\leq d_\Gamma\br{\vbeta}<2^{k_0+1}t_0\). By construction, either 
\begin{equation*}
    \br{\alpha,\vbeta}\in \bigcup_{k=0}^{k_{0}}\bigcup_{T\in \T_{k}}D_{T}\subset \bigcup_{\br{I,\vgamma}\in\T}D_T
\end{equation*}
and the covering property \eqref{select1-1} is verified, or the alternative \(\br{\alpha,\vbeta}\in \P_{k_0}\) happens. In the case where \(\br{\alpha,\vbeta}\in \P_{k_0}\), the maximality of \(\P^\ast_{k_0}\) guarantees the existence of a point \(\br{\alpha',\vbeta'}\in \P_{k_0}^\ast\) such that
\begin{equation*}
    R_{\alpha,\vbeta,2^{k_0}t_0}\cap R_{\alpha',\vbeta',2^{k_0}t_0}\neq \varnothing.
\end{equation*}
As a direct consequence
\begin{equation*}
    \br{\alpha,\gamma\br{\vbeta}_j}\in \br{\alpha'+\frac{3c_s}{2^{k_0}t_0}\Br{-\frac{1}{2},\frac{1}{2}}} \times \br{\gamma\br{\vbeta'}_j+3c_f 2^{k_0}t_0\Br{-\frac{1}{2},\frac{1}{2}}}.
\end{equation*}
We now recall that \(M \geq \frac{3c_f}{2\delta_0\br{1-\delta_1}}\). This implies that \(\vgamma\br{\vbeta}\in \Gamma^i_{\vbeta'}\) for some \(i\in \Br{-M,M}\).
In particular, \(\gamma^{\br{i}}_-\br{\vbeta}_j\leq \gamma\br{\vbeta}_j\leq \gamma^{\br{i}}_+\br{\vbeta}_j\leq \gamma^{\br{i}}_-\br{\vbeta}_j +\delta_0 \br{1-\delta_1} 2^{k_0}t_0\) for the same \(i\).
Using statement \textit{(1)} in Proposition \ref{proptentgeo} and the fact that \(c\geq \frac{1}{\delta_1}\), we obtain
\begin{equation*}
    \br{\alpha,\vbeta}\in \bigcup_{\br{I,\vgamma}\in \T_{k_0+1}}D_T\subset \bigcup_{\br{I,\vgamma}\in \T}D_T.
\end{equation*}
The estimate \eqref{select1-2} holds directly by the construction of $\P$ and $\T$. We now verify the orthogonality property for $\P$. For $(\alpha,\vbeta)\in P_{k}^{\ast}$ and  $(\alpha',\vbeta')\in \P_{k'}^{\ast}$, we split the argument into two cases.

case 1: $k=k'$. By \eqref{Rdijonit63}, we have either
\begin{equation}\label{eq67}
    |\alpha-\alpha'|\geq c_{s}(2^{k}t_{0})^{-1}
\end{equation}
or 
\begin{equation}\label{eq68}
    |\gamma(\vbeta)_{j}-\gamma(\vbeta')_{j}|\geq c_{f}2^{k}t_{0}.
\end{equation}
In case \eqref{eq67}, taking $c_{s}=4$ gives
\begin{equation*}
     |\alpha-\alpha'|\geq \frac{c_{s}}{2}(d_{\Gamma}(\vbeta)^{-1}+d_{\Gamma}(\vbeta')^{-1})\geq 2(d_{\Gamma}(\vbeta)^{-1}+d_{\Gamma}(\vbeta')^{-1}).
\end{equation*}
In case \eqref{eq68}, taking $c_{f}=\frac{11}{\delta_{2}}$ gives $c_{f}\geq 8\geq 4( \rho +1)$, and thus
\begin{equation*}
    |\beta_{j}-\beta^{'}_{j}|\geq |\gamma(\vbeta)_{j}-\gamma(\vbeta')_{j}|-|\beta_{j}-\gamma(\vbeta)_{j}|-|\beta^{'}_{j}-\gamma(\vbeta')_{j}|
\end{equation*}
\begin{equation*}
    \geq c_{f}2^{k}t_{0}-d_{\Gamma}(\vbeta)-d_{\Gamma}(\vbeta') \geq (\frac{c_{f}}{4}-1)(d_{\Gamma}(\vbeta)+d_{\Gamma}(\vbeta'))\geq  \rho (d_{\Gamma}(\vbeta)+d_{\Gamma}(\vbeta')).
\end{equation*}

case 2: $k<k'$. Either 
\begin{equation}\label{eq613}
    \alpha' \notin \alpha+\frac{c}{2^{k}t_{0}}\Br{-\frac{1}{2},\frac{1}{2}},
\end{equation}
or 
\begin{equation}\label{eq614}
      \alpha' \in \alpha+\frac{c}{2^{k}t_{0}}\Br{-\frac{1}{2},\frac{1}{2}}.
\end{equation}
In the case \eqref{eq613}, since $c\geq 3c_{s}\geq 12$,
\begin{equation*}
    |\alpha-\alpha'|\geq \frac{1}{2}\cdot \frac{c}{ 2^k t_0}\geq \frac{c}{4}(\frac{1}{2^{k}t_{0}}+\frac{1}{2^{k'}t_{0}})\geq  2(d_{\Gamma}(\vbeta)^{-1}+d_{\Gamma}(\vbeta')^{-1}).
\end{equation*}
In the case \eqref{eq614}, notice that
\begin{equation*}
    (\alpha',\vbeta')\notin  \bigg( \bigcup_{k=0}^{k'}\bigcup_{T\in \T_{k}}D_{T}\bigg).
\end{equation*}
In particular, since $k<k'$,
\begin{equation*}
      (\alpha',\vbeta')\notin \bigcup_{i=-M}^{M}\left(\alpha+\frac{c}{2^{k}t_{0}}[-\frac{1}{2},\frac{1}{2}]\right) \times W_{\vgamma_{\pm}^{(i)}(\vbeta),\frac{2^{k}t_{0}}{c}}.
\end{equation*}
Together with \eqref{eq614}, and by $c\geq \frac{1}{\delta_1}$ and Proposition \ref{proptentgeo}, we obtain
\begin{equation}\label{eq618}
    \vbeta' \notin \bigcup_{\substack{\vgamma \in \Gamma\\ |\gamma_{j}-\gamma(\vbeta)_{j}|\leq \frac{3}{2}c_{f}2^{k}t_{0}}}W_{\vgamma, 2^{k}t_{0}}.
\end{equation}
On the other hand,
\begin{equation*}
    \vbeta' \in W_{\vgamma(\vbeta'),2^{k'}t_{0}}\subseteq  W_{\vgamma(\vbeta'),2^{k}t_{0}}.
\end{equation*}
Together with \eqref{eq618},  we have
\begin{equation}\label{eq620}
    |\vgamma(\vbeta')-\vgamma(\vbeta)|\geq |\gamma(\vbeta')_{j}-\gamma(\vbeta)_{j}|\geq \frac{3}{2}c_{f}2^{k}t_{0}.
\end{equation}
Note that $\vbeta' \notin W_{\vgamma (\vbeta),2^{k}t_{0}}$ and
\begin{equation*}
    |\vbeta'-\vgamma (\vbeta)|\geq d_{\Gamma}(\vbeta')\geq d_{\Gamma}(\vbeta)\geq 2^{k}t_{0},
\end{equation*}
we have $\vbeta' \notin W_{\vgamma(\vbeta),0}$. That is,
\begin{equation}\label{eq622}
    \delta_{1}|\vbeta'-\vgamma(\vbeta)|\geq d_{\Gamma}(\vbeta')=|\vbeta'-\vgamma(\vbeta')|.
\end{equation}
Combining \eqref{eq620} and 
\eqref{eq622},
\begin{equation}\label{eq623}
    \frac{3}{2}c_{f}2^{k}t_{0}\leq |\vgamma(\vbeta')-\vgamma(\vbeta)|\leq |\vbeta'-\vgamma (\vbeta)|+ |\vbeta'-\vgamma (\vbeta')|\leq (1+\delta_{1})|\vbeta'-\vgamma (\vbeta)|.
\end{equation}
Since $\vbeta' \notin W_{\vgamma(\vbeta),0}$, by Lemma \ref{UinW}, we also have $\vbeta' \notin U_{\vgamma(\vbeta)}^{j}$. Therefore,
\begin{equation}\label{eq624}
    |\beta_{j}^{'}-\beta_{j}|\geq |\beta_{j}^{'}-\gamma(\vbeta)_{j}|-|\beta_{j}-\gamma(\vbeta)_{j}|\geq \delta_{2}|\vbeta'-\vgamma (\vbeta)|-2^{k+1}t_{0}.
\end{equation}
Together with \eqref{eq623} and the trivial estimate $|\vbeta'-\vgamma(\vbeta)|\geq d_{\Gamma}(\vbeta')$, \eqref{eq624} can be estimated from below by
\begin{equation*}
     \frac{\delta_{2}}{2}\frac{3c_{f}}{2(1+\delta_{1})}2^{k}t_{0}+\frac{\delta_{2}}{2}d_{\Gamma}(\vbeta')-2^{k+1}t_{0}
\end{equation*}
\begin{equation*}
    \geq \left(\frac{3\delta_{2}c_{f}}{8(1+\delta_{1})}-1 \right)d_{\Gamma}(\vbeta)+\frac{\delta_{2}}{2}d_{\Gamma}(\vbeta')\geq  \rho (d_{\Gamma}(\vbeta)+d_{\Gamma}(\vbeta')).
\end{equation*}
This completes the proof of Proposition \ref{prop_select}.

\section{Proof of Proposition \ref{prop_selectL2}: Selection Algorithm, $L^{2}$ Component}
 Since $\Omega$ is compact, we may assume $\Omega \subseteq \R\times P_{j}^{-1}([-A,A])$. We set up an iteration algorithm. Let $\Omega_{0}:=\Omega$ and \(t_0=\frac{\lambda^2}{2C^2\nrm{f}_{L^2}^2}\), where \(C=C\br{\theta_0}\) is the constant that realizes the estimate \eqref{eq_global_L2_g_est}
 \begin{equation}
     \nrm{F_j f}_{L^2_\nu \br{\R\times \br{W_{\vgamma,0}\setminus U_\vgamma^j}}}\leq C\nrm{f}_{L^2}.
 \end{equation}
 Suppose now $\Omega_{k-1}$ is given, we define a collection of intervals $\I_{k}$ for integer \(k\) in the range $0\leq k\leq \frac{2A}{\delta_{0}(1-\delta_1)t_{0}}+1$. The collection \(\I_k\) consists of interval $I$ with the following properties: there is a point $\vgamma$ in the strip
\begin{equation}\label{strip67}
  \Gamma^k:=\Gamma\cap P_{j}^{-1}\bigg(-A+\delta_{0}(1-\delta_1)t_{0}\cdot[k-1,k]\bigg)\subseteq V
\end{equation}
such that
\begin{equation}\label{largesize68}
  |I|^{-\frac{1}{2}}\|1_{\Omega_{k-1}} F_jf
            \|_{L^2_{\nu}\left( I\times (W_{\vgamma,1/\abs{I}}\setminus U_{\vgamma}^{j}\right)^{<j})}\geq \frac{\lambda}{\sqrt{2}}. 
\end{equation}
We apply Vitali covering lemma on \(5\I_k:=\left\{5I:I\in \I_k\right\}\). As a result, there is a subcollection $\J_{k}\subseteq \I_{k}$ such that for all distinct $I,J\in \J_{k}$,
\begin{equation*}
    5I\cap 5J=\varnothing
\end{equation*}
and 
\begin{equation*}
    \bigcup \I_{k}\subseteq \bigcup 5\I_{k} \subseteq \bigcup 25\J_{k}.
\end{equation*}
For $I\in \J_{k}$, let $\vgamma_{I}\in\Gamma^k$ be the point \(\gamma\) that realizes \eqref{largesize68} and \(\vgamma^k_+,\vgamma^k_-\in \Gamma^k\) be the two endpoints such that \(P_j\br{\Gamma^k}\subseteq\Br{\gamma^k_{-,j},\gamma^k_{+,j}}\). Define
\begin{equation}
    \S_{k}:=\left\{\left( I,\vgamma_{I},\Omega_{k-1}\cap I\times \left(W_{\vgamma,\frac{1}{\abs{I}}}\setminus U_{\vgamma}^{j}\right)^{<j}  \right)\: : \:I\in \J_{k}  \right\}
\end{equation}
% For all $\vgamma \in \Gamma^k$ and $I \in \I_k$, since
% \begin{equation}
%     \gamma^k_{-,j}\leq \gamma_j\leq \gamma^k_{+,j}\leq
%     \gamma^k_{-,j}+\delta_0\br{1-\delta_1}t_0\leq
%     \gamma^k_{-,j}+\delta_0\br{1-\delta_1}\frac{1}{\abs{I}},
% \end{equation}
% by statement \textit{(1)} in Proposition \ref{proptentgeo} and the construction of $\J_{k}$, there exist a $J\in \J_{k}$ such that
% \begin{equation}
%     I\subseteq 5I\subseteq 25J\subseteq \frac{25}{\delta_1}J
% \end{equation}
% and 
% \begin{equation}
%     W_{\vgamma,\frac{1}{|I|}}\subseteq W_{\vgamma_{-}^{k},\frac{\delta_1}{|I|}}\cup W_{\vgamma_{+}^{k},\frac{\delta_1}{|I|}}\subseteq W_{\vgamma_{-}^{k},\frac{\delta_1}{|25J|}}\cup W_{\vgamma_{+}^{k},\frac{\delta_1}{|25J|}}.
% \end{equation}
and
\begin{equation}
    \T_{k}:=\left\{\left(\frac{25}{\delta_1}I,\vgamma_{\pm}^{k}\right) \: : \: I\in \J_{k}    \right\}\cup \left\{\left( \frac{1}{\delta_1}I,\vgamma_{I}\right) :I \in \J_{k}   \right\}.
\end{equation}
For the next iteration, we set
\begin{equation}
    \Omega_{k}:=\Omega_{k-1}\setminus \left( \bigcup_{i=1}^{k}\left(\bigcup_{T\in \T_{i}} D_{T}\right) \right).
\end{equation}
Eventually, we obtain
\begin{equation}
    \S:= \bigcup_{k=1}^{\infty}\S_{k}, \quad \T:= \bigcup_{k=1}^{\infty}\T_{k}.
\end{equation}
We will show that $\S$ and $\T$ satisfy the desired properties in Proposition \ref{prop_selectL2}. To show the covering property
\eqref{select2-1}, we assume the alternative that there is a tent \(\br{I,\vgamma}\) that violates \eqref{select2-1}:
\begin{equation}
    |I|^{-\frac{1}{2}}\nrm{
        1_{\Omega\setminus\bigcup_{T\in \T}D_{T}} F_jf
    }_{L^2_{\nu}\br{ I\times (W_{\vgamma,1/|I|}\setminus U_{\vgamma}^{j})^{<j}}}>\frac{\lambda}{\sqrt{2}}.
\end{equation}
By construction, there is a \(k\) such that $I \in \I_k$ and \(\vgamma\in\Gamma^k\) realizes \eqref{largesize68}. The relation \eqref{largesize68} and \eqref{eq_global_L2_g_est} imply the following estimate
\begin{equation}\label{eq711}
    \frac{\lambda}{\sqrt{2}} \leq \abs{I}^{-\frac{1}{2}}\nrm{F_j f}_{L^2_\nu \br{\R\times \br{W_{\vgamma,0}\setminus U_\vgamma^j}}} \leq C \abs{I}^{-\frac{1}{2}}\nrm{f}_{L^2}.
\end{equation}
By the definition of \(t_0\), $\|f\|_{L^{2}}=\frac{\lambda}{C\sqrt{2t_0}}$, and thus \eqref{eq711} is equivalent to \(t_0\leq \frac{1}{\abs{I}}\). As a result,
\begin{equation*}
    \gamma^k_{-,j}\leq \gamma_j\leq \gamma^k_{+,j}\leq
    \gamma^k_{-,j}+\delta_0\br{1-\delta_1}t_0\leq
    \gamma^k_{-,j}+\delta_0\br{1-\delta_1}\frac{1}{\abs{I}}.
\end{equation*}
By statement \textit{(1)} in Proposition \ref{proptentgeo} and the construction of $\J_{k}$, there exist a $J\in \J_{k}$ such that
\begin{equation*}
    I\subseteq 5I\subseteq 25J\subseteq \frac{25}{\delta_1}J
\end{equation*}
and 
\begin{equation*}
    W_{\vgamma,\frac{1}{|I|}}\subseteq W_{\vgamma_{-}^{k},\frac{\delta_1}{|I|}}\cup W_{\vgamma_{+}^{k},\frac{\delta_1}{|I|}}\subseteq W_{\vgamma_{-}^{k},\frac{\delta_1}{|25J|}}\cup W_{\vgamma_{+}^{k},\frac{\delta_1}{|25J|}}.
\end{equation*}
That is, 
\begin{equation*}
    I\times (W_{\vgamma,1/|I|}\setminus U_{\vgamma}^{j})^{<j}\subseteq
    D_{\br{I,\vgamma}}\subseteq \bigcup_{T\in \T_k }D_T
    \subseteq \bigcup_{T\in \T }D_T.
\end{equation*}
This is a contradiction, and thus \eqref{select2-1} must hold.
The estimate \eqref{select2-2} follows directly from the construction of $\S$ and $\T$. In the following, we check $\S$ satisfies the orthogonal property. Given $(I,\vgamma,S)\in \S_{k}, (\alpha,\vbeta)\in S$ and $(I',\vgamma',S')\in \S_{k'}, (\alpha',\vbeta')\in S'$, without loss of generality, we may assume $\vgamma_{j}\leq \vgamma_{j'}$. We split into two cases according to whether they are in the same strip.

case 1: $k=k'$. By construction, we have $S\subseteq I\times V\setminus \Gamma$, $S'\subseteq I'\times V\setminus \Gamma$, and $5I\cap 5I'=\varnothing$. Hence
\begin{equation*}
    |\alpha-\alpha'|\geq \frac{5-1}{2}(|I|+|I'|)\geq 2|I|.
\end{equation*}

case 2: $k<k'$. By the construction of $S_{k'}$,
\begin{equation}\label{setminus617}
    S'\subseteq I'\times \left(W_{\vgamma',\frac{1}{\abs{I'}}}\setminus U_{\vgamma'}^{j}\right)^{<j}\cap \Omega_{k'-1}\subseteq  I'\times \left(W_{\vgamma',\frac{1}{\abs{I'}}}\setminus U_{\vgamma'}^{j}\right)^{<j} \setminus \left(\frac{1}{\delta_1}I\times W_{\vgamma,\frac{\delta_1}{|I|}}\right).
\end{equation}
We either have
\begin{equation*}
    |\alpha-\alpha'|\geq \frac{1/\delta_1 -1}{2}|I|\geq 2\abs{I},
\end{equation*}
and the orthogonality property \eqref{select2-31} is verified, or the alternative
\begin{equation*}
    |\alpha-\alpha'|< \frac{1/\delta_1 -1}{2}|I|.
\end{equation*}
Since $\alpha \in I$, we have $\alpha' \in \frac{1}{\delta_1}I$. Then by \eqref{setminus617}, we have
\begin{equation*}
    \vbeta' \in \left(W_{\vgamma',\frac{1}{\abs{I'}}}\setminus U_{\vgamma'}^{j}\right)^{<j} \setminus W_{\vgamma,\frac{\delta_1}{|I|}}.
\end{equation*}
By applying the statement \textit{(2)} of Proposition 
\ref{proptentgeo}, we obtain \eqref{select2-32}. This completes the proof of Proposition \ref{prop_selectL2}.

\section{Proof of Proposition \ref{prop_L2_emb}: Bessel Type Estimate}\label{secbessel}
Throughout the rest of the argument, we take \(\varepsilon\) to be the specific value $\frac{\delta_{1}\rho^{2}}{2}$.
\begin{lemma}\label{prop_Linfty_comp_est}
    Let \(\P\) and \(f\) be as in Proposition \ref{prop_select}. We have
    \begin{equation*}
        \sum_{\br{\alpha,\vbeta}\in \P} \frac{\abs{F_j f\br{\alpha,\vbeta}}^2}{d_\Gamma\br{\vbeta}}\lesssim \nrm{f}^2_{L^2}.
    \end{equation*}
\end{lemma}
\begin{proof}
 Direct calculation and Cauchy-Schwarz yield
 \begin{equation*}
        \sum_{\br{\alpha,\vbeta}\in \P} 
            \frac{
                \abs{
                    F_j f\br{\alpha,\vbeta}
                }^2
            }{
                d_\Gamma\br{\vbeta}
            }
        = 
        \sum_{\br{\alpha,\vbeta}\in \P} 
            \frac{
                    \ang{
                        f,
                        \varphi_{\alpha,\vbeta}^{j}
                    }
                \overline{
                    F_j f\br{\alpha,\vbeta}
                }
            }{
                d_\Gamma\br{\vbeta}
            }
\end{equation*}
\begin{equation*}
        = 
        \ang{
            f,
            \sum_{\br{\alpha,\vbeta}\in \P} 
                \frac{
                    F_j f\br{\alpha,\vbeta}
                    \varphi_{\alpha,\vbeta}^{j}
                }{
                    d_\Gamma\br{\vbeta}
                }
        }
        \leq 
        \nrm{f}_{L^2}
        \nrm{
            \sum_{\br{\alpha,\vbeta}\in \P} 
                \frac{
                    F_j f\br{\alpha,\vbeta}
                    \varphi_{\alpha,\vbeta}^{j}
                }{
                    d_\Gamma\br{\vbeta}
                }
        }_{L^2}.
\end{equation*}
    It suffices to show
    \begin{equation*}
        \nrm{
            \sum_{\br{\alpha,\vbeta}\in \P} 
                \frac{
                    F_j f\br{\alpha,\vbeta}
                    \varphi_{\alpha,\vbeta}^{j}
                }{
                    d_\Gamma\br{\vbeta}
                }
        }_{L^2}^2
        \lesssim
        \sum_{\br{\alpha,\vbeta}\in \P} 
            \frac{
                \abs{
                    F_j f\br{\alpha,\vbeta}
                }^2
            }{
                d_\Gamma\br{\vbeta}
            }.
    \end{equation*}
    We develop the \(L^2\) norm and introduce asymmetry
\begin{equation*}
        \nrm{
            \sum_{\br{\alpha,\vbeta}\in \P} 
                \frac{
                    F_j f\br{\alpha,\vbeta}
                    \varphi_{\alpha,\vbeta}^{j}
                }{
                    d_\Gamma\br{\vbeta}
                }
        }_{L^2}^2
        \leq 
        \sum_{
            \br{\alpha,\vbeta},
            \br{\alpha',\vbeta'}\in \P
        }
            \abs{
                F_j f\br{\alpha,\vbeta}
            }
            \abs{
                F_j f\br{\alpha',\vbeta'} 
            }
            \frac{
                \abs{
                    \ang{
                        \varphi_{\alpha,\vbeta}^{j},
                        \varphi_{\alpha',\vbeta'}^{j}
                    }
                }
            }{d_\Gamma\br{\vbeta} d_\Gamma\br{\vbeta'}}
\end{equation*}
\begin{equation*}    
\leq 
        2
        \sum_{
            \substack{
                \br{\alpha,\vbeta},
                \br{\alpha',\vbeta'}\in \P\\
                d_\Gamma\br{\vbeta} \leq d_\Gamma\br{\vbeta'}
            }
        }
            \abs{
                F_j f\br{\alpha,\vbeta}
            }
            \abs{
                F_j f\br{\alpha',\vbeta'} 
            }
            \frac{
                \abs{
                    \ang{
                        \varphi_{\alpha,\vbeta}^{j},
                        \varphi_{\alpha',\vbeta'}^{j}
                    }
                }
            }{d_\Gamma\br{\vbeta} d_\Gamma\br{\vbeta'}}.
\end{equation*}   
    We recall that \(\P\subset \abs{F_j f}^{-1}\left(\lambda,2\lambda\right]\). In particular,
    \begin{equation*}
        \forall \br{\alpha,\vbeta},\br{\alpha',\vbeta'}\in \P,\quad
        \abs{
            F_j f\br{\alpha',\vbeta'}
        }
        \sim
        \abs{
            F_j f\br{\alpha,\vbeta}
        }.
    \end{equation*}
    Consequently, we may perform the following substitution
    \begin{equation*}
        \nrm{
            \sum_{\br{\alpha,\vbeta}\in \P} 
                \frac{
                    F_j f\br{\alpha,\vbeta}
                    \varphi_{\alpha,\vbeta}^{j}
                }{
                    d_\Gamma\br{\vbeta}
                }
        }_{L^2}^2
        \lesssim
        \sum_{
            \br{\alpha,\vbeta}\in \P
        }
            \frac{
                \abs{F_j f\br{\alpha,\vbeta}}^2
            }{
                d_\Gamma\br{\vbeta}
            }
            \sum_{
                \substack{
                    \br{\alpha',\vbeta'}\in \P\\
                    d_\Gamma\br{\vbeta}\leq d_\Gamma\br{\vbeta'}
                }
            }
            \frac{
                \abs{
                    \ang{
                        \varphi_{\alpha,\vbeta}^{j},
                        \varphi_{\alpha',\vbeta'}^{j}
                    }
                }
            }{d_\Gamma\br{\vbeta'}}.
    \end{equation*}
    It remains to show, for \(\br{\alpha,\vbeta}\in \P\),
    \begin{equation}
        \sum_{
                \substack{
                    \br{\alpha',\vbeta'}\in \P\\
                    d_\Gamma\br{\vbeta}\leq d_\Gamma\br{\vbeta'}
                }
            }
            \frac{
                \abs{
                    \ang{
                        \varphi_{\alpha,\vbeta}^{j},
                        \varphi_{\alpha',\vbeta'}^{j}
                    }
                }
            }{d_\Gamma\br{\vbeta'}}\lesssim 1.
    \end{equation}
    We discard all terms with no contribution to the sum. Therefore, we may assume
    \begin{equation*}
        \ang{
            \varphi_{\alpha,\vbeta}^{j},
            \varphi_{\alpha',\vbeta'}^{j}
        }\neq 0.
    \end{equation*}
    This implies that the frequency supports of the two functions overlap. In other words,
    \begin{equation*}
        \abs{
            \beta_j-\beta_j'
        }\leq \frac{4\varepsilon}{10} \br{d_\Gamma\br{\vbeta}+d_\Gamma\br{\vbeta'}}\leq \frac{8\varepsilon}{10}d_{\Gamma}(\vbeta')< \rho d_{\Gamma}(\vbeta').
    \end{equation*}
    As a result, for any other \(\br{\alpha'',\vbeta''}\) remained in the summand,
\begin{equation*}
        \abs{
            \beta_j'-\beta_j''
        }\leq
        \abs{
            \beta_j-\beta_j'
        }+
        \abs{
            \beta_j-\beta_j''
        }
        <
       \rho\br{d_\Gamma\br{\vbeta'}+d_\Gamma\br{\vbeta''}}.
\end{equation*}
    This violates \eqref{select1-32}. Thus, for distinct \(\br{\alpha',\vbeta'},\br{\alpha'',\vbeta''}\) remained in the summand, they must satisfy \eqref{select1-31}
    \begin{equation*}
        \abs{\alpha'-\alpha''}\geq 2 \br{d_\Gamma\br{\vbeta'}^{-1}+d_\Gamma\br{\vbeta''}^{-1}}.
    \end{equation*}
    Finally, we apply the standard wave-packet estimate as in \textit{Lemma 2.1} of \cite{thiele2006wave} and utilize the physical separation to complete the proof
    \begin{equation*}
        \sum_{
            \substack{
                \br{\alpha',\vbeta'}\in \P\\
                d_\Gamma\br{\vbeta}\leq d_\Gamma\br{\vbeta'}
            }
        }
            \frac{
                \abs{
                    \ang{
                        \varphi_{\alpha,\vbeta}^{j},
                        \varphi_{\alpha',\vbeta'}^{j}
                    }
                }
            }{d_\Gamma\br{\vbeta'}}
            \lesssim \!\!\!
        \sum_{
            \substack{
                \br{\alpha',\vbeta'}\in \P\\
                d_\Gamma\br{\vbeta}\leq d_\Gamma\br{\vbeta'}
            }
        }
            \frac{d_\Gamma\br{\vbeta}}{d_\Gamma\br{\vbeta'}}
            \br{1+d_\Gamma\br{\vbeta}\abs{\alpha-\alpha'}}^{-N}
     \end{equation*}
    \begin{equation*}
        \sim 
        1+
        \sum_{
            \substack{
                \br{\alpha',\vbeta'}\in \P\setminus\BR{\br{\alpha,\vbeta}}\\
                d_\Gamma\br{\vbeta}\leq d_\Gamma\br{\vbeta'}
            }
        }
            \int_{\alpha'+\frac{\Br{-1,1}}{d_\Gamma\br{\vbeta'}}}
                d_\Gamma\br{\vbeta}
                \br{1+d_\Gamma\br{\vbeta}\abs{\alpha-x}}^{-N}
            dx
    \end{equation*}
    \begin{equation*}
        \lesssim 
        1+
        \int_{\R}
            d_\Gamma\br{\vbeta}
            \br{1+d_\Gamma\br{\vbeta}\abs{\alpha-x}}^{-N}
        dx
        =
        1+\int_\R \br{1+\abs{x}}^{-N}dx
        \sim 1.
     \end{equation*}
\end{proof}

\begin{lemma}\label{lem_L2_est_lac_scale_bd}
    We define a constant
    \begin{equation}
        c:=\frac{\frac{4\varepsilon}{10}+\frac{1}{\delta_1}}{\delta_2-\frac{4\varepsilon}{10}}<\frac{10\rho}{4\epsilon}.
    \end{equation}
    Let \(\vgamma,\vgamma'\in\Gamma\) satisfy \(\gamma_j\leq\gamma'_j\). For \(\vbeta\in \br{W_{\vgamma,0}\setminus U_{\vgamma}^j}^{<j}\) and \(\vbeta'\in  W_{\vgamma',0}\), suppose
    \begin{equation}
        \abs{\beta_j-\beta_j'}\leq \frac{4\varepsilon}{10}\br{d_\Gamma\br{\vbeta}+d_\Gamma\br{\vbeta'}},
    \end{equation}
    then
    \begin{equation}
        d_\Gamma\br{\vbeta}\leq c d_\Gamma\br{\vbeta'}.
        \label{eq_lac_scale_lbd}
    \end{equation}
    If additionally that \(\vbeta'\notin U_{\vgamma'}^j\),
    \begin{equation}
        \br{\delta_2-\frac{4\varepsilon }{10}\br{1+c}}d_\Gamma\br{\vbeta'}
        \leq 
        \abs{\beta_j-\gamma_j'}
        \leq
        \br{\frac{1}{\delta_1}+\frac{4\varepsilon}{10}\br{1+c}}d_\Gamma\br{\vbeta'}.
        \label{eq_lac_scale_eqbd}
    \end{equation}
\end{lemma}

\begin{proof}
    We first show \eqref{eq_lac_scale_lbd} via estimating \(\abs{\beta_j-\gamma_j'}\) from above and from below. On the one hand, the triangle inequality gives
    \begin{equation}
        \abs{\beta_j-\gamma_j'}\leq 
        \abs{\beta_j-\beta_j'}+
        \abs{\beta_j'-\gamma_j'}.\label{eq_lac_scale_lbd_tria}
    \end{equation}
    The point \(\vbeta'\in W_{\vgamma',0}\) satisfies
    \begin{equation}
        \abs{\beta_j'-\gamma_j'}\leq\abs{\vbeta'-\vgamma'}\leq \frac{1}{\delta_1}d_\Gamma\br{\vbeta'},
        \label{eq_lac_scale_lbd_ubd}
    \end{equation}
    which controls the second term on the right-hand side of \eqref{eq_lac_scale_lbd_tria}. We   control the first term on the right-hand side of \eqref{eq_lac_scale_lbd_tria}
    \begin{equation*}
        \abs{\beta_j-\beta_j'}\leq \frac{4\varepsilon}{10}\br{d_\Gamma\br{\vbeta}+d_\Gamma\br{\vbeta'}}.
    \end{equation*}
   Hence, we obtain the following
    \begin{equation}\label{eq923}
        \abs{\beta_j-\gamma'_j}\leq \frac{4\varepsilon}{10} d_\Gamma\br{\vbeta}+\br{\frac{4\varepsilon}{10}+\frac{1}{\delta_1}}d_\Gamma\br{\vbeta'}.
    \end{equation}
    On the other hand, by assumption, the three points \(\vbeta,\vgamma,\vgamma'\) satisfy the ordering relation \(\beta_j\leq \gamma_j\leq \gamma_j'\), and thus,
    \begin{equation}\label{eq924}
        \abs{\beta_j-\gamma_j'}=\abs{\beta_j-\gamma_j}+\abs{\gamma_j-\gamma_j'}\geq \abs{\beta_j-\gamma_j}.
    \end{equation}
    Since \(\vbeta\in \br{W_{\vgamma,0}\setminus U_{\vgamma}^j}^{<j}\subseteq \br{U_\vgamma^j}^c\), the above quantity can be estimated from below by
    \begin{equation}\label{eq925}
       \delta_2\abs{\vbeta-\vgamma}\geq \delta_2 d_\Gamma\br{\vbeta}.
    \end{equation}
    Combining \eqref{eq923}, \eqref{eq924}, and \eqref{eq925}, we obtain
    \begin{equation*}
        \delta_2 d_\Gamma\br{\vbeta}\leq 
        \abs{\beta_j-\gamma_j'}
        \leq \frac{4\varepsilon}{10} d_\Gamma\br{\vbeta}+\br{\frac{4\varepsilon}{10}+\frac{1}{\delta_1}}d_\Gamma\br{\vbeta'}.
    \end{equation*}
    This completes the proof for \eqref{eq_lac_scale_lbd}. To show \eqref{eq_lac_scale_eqbd}, we further assume \(\vbeta'\notin U_{\gamma'}^j\) and obtain a lower bound similar to \eqref{eq_lac_scale_lbd_ubd}
    \begin{equation*}
        \delta_2 d_\Gamma\br{\vbeta'}\leq \delta_2\abs{\vbeta'-\vgamma'}\leq \abs{\beta_j'-\gamma_j'}.
    \end{equation*}
   By triangle inequality 
   \begin{equation*}
   \abs{\beta_j'-\gamma_j'}-\abs{\beta_j-\beta_j'}\leq \abs{\beta_j-\gamma_j'}\leq \abs{\beta'_j-\gamma_j'}+\abs{\beta_j-\beta_j'}.
   \end{equation*}
   Hence, we obtain
    \begin{equation*}
        \delta_2 d_\Gamma\br{\vbeta'}-\abs{\beta_j-\beta_j'}\leq \abs{\beta_j-\gamma_j'} \leq \frac{1}{\delta_1}d_\Gamma\br{\vbeta'}+\abs{\beta_j-\beta_j'}.
    \end{equation*}
    It remains to control \(\abs{\beta_j-\beta_j'}\) from above with a small constant multiple of \(d_\Gamma\br{\vbeta'}\), which can be achieved via utilizing the consequence of \eqref{eq_lac_scale_lbd}
    \begin{equation*}
        \abs{\beta_j-\beta_j'}\leq \frac{4\varepsilon}{10} \br{d_\Gamma\br{\vbeta}+d_\Gamma\br{\vbeta'}}\leq \frac{4\varepsilon}{10}\cdot\br{1+c}d_\Gamma\br{\vbeta'}.
    \end{equation*}
\end{proof}

\begin{lemma}\label{prop_L2_comp_est}
    Let \(\Omega,\lambda,\S\) and \(f\) be as in the Proposition \ref{prop_selectL2}. Assume $  1_\Omega\abs{F_jf}\leq \lambda$, then we have the following bound
    \begin{equation}
        \sum_{\br{I,\vgamma,S}\in\S} \nrm{F_jf}_{L^2_{\nu}\br{S}}^2\lesssim \nrm{f}^2_{L^2}.
    \end{equation}
\end{lemma}

\begin{proof}
    The following argument is parallel to the first parts of the proof of Lemma \ref{prop_Linfty_comp_est}
    \begin{equation*}
        \sum_{
            \br{I,\vgamma,S}\in \S
        }\nrm{F_j f}_{L^2_{\nu}\br{S}}^2
        = 
        \sum_{
            \br{I,\vgamma,S}\in \S
        }
            \int_S
                \ang{
                    f,
                    \varphi_{\alpha,\vbeta}^{j}
                }
                \overline{F_j f\br{\alpha,\vbeta}}
            d\alpha
            d\mu\br{\vbeta}
    \end{equation*}
    \begin{equation*}
        = 
        \ang{
            f,
            \sum_{
                \br{I,\vgamma,S}\in \S
            }
                \int_S
                    F_j f\br{\alpha,\vbeta}
                    \varphi_{\alpha,\vbeta}^{j}
                d\alpha d\mu\br{\vbeta}
        }
\end{equation*}
\begin{equation*}
        \leq 
        \nrm{f}_{L^2}
        \nrm{
            \sum_{
                \br{I,\vgamma,S}\in \S
            }
                \int_S
                    F_j f\br{\alpha,\vbeta}
                    \varphi_{\alpha,\vbeta}^{j}
                d\alpha d\mu\br{\vbeta}
        }_{L^2}.
    \end{equation*}
    Again, it suffices to show
    \begin{equation}\label{eq935}
        \nrm{
            \sum_{
                \br{I,\vgamma,S}\in \S
            }
                \int_S
                    F_j f\br{\alpha,\vbeta}
                    \varphi_{\alpha,\vbeta}^{j}
                d\alpha d\mu\br{\vbeta}
        }_{L^2}^2
        \lesssim
        \sum_{
            \br{I,\vgamma,S}\in \S
        }\nrm{F_j f}_{L^2_{\nu}\br{S}}^2.
    \end{equation}
   Developing the \(L^2\) norm, we can estimate the left-hand side of \eqref{eq935} by
    \begin{equation*}
        \leq 
        \sum_{
            \substack{
                \br{I,\vgamma,S}\in\S\\
                \br{I',\vgamma',S'}\in\S
            }
        }
            \int_{S}
                \abs{F_j f\br{\alpha,\vbeta}}
                \int_{S'}
                    \abs{F_j f\br{\alpha',\vbeta'}}
                    \cdot\abs{
                        \ang{
                            \varphi_{\alpha,\vbeta}^{j},
                            \varphi_{\alpha',\vbeta'}^{j}
                        }
                    }
                d\alpha' d\mu\br{\vbeta'}
            d\alpha d\mu\br{\vbeta}.
     \end{equation*}
    The summation can be further decomposed into diagonal terms
    \begin{equation}\label{eq_L2_est_diag}
        \sum_{
            \br{I,\vgamma,S}\in\S
        }
            \int_{S}
                \int_{S}
                    \abs{F_j f\br{\alpha,\vbeta}}
                    \abs{F_j f\br{\alpha',\vbeta'}}
            \cdot\abs{
                        \ang{
                            \varphi_{\alpha,\vbeta}^{j},
                            \varphi_{\alpha',\vbeta'}^{j}
                        }
                    }
                d\alpha' d\mu\br{\vbeta'}
            d\alpha d\mu\br{\vbeta}.
     \end{equation}
    and two copies of off-diagonal terms
    \begin{equation}\label{eq_L2_est_off_diag}
        \sum_{
            \substack{
                \br{I,\vgamma,S}\in\S\\
                \br{I',\vgamma',S'}\in\S\\
                \br{I,\vgamma,S}\neq
                \br{I',\vgamma',S'}\\
                \gamma_j\leq \gamma_j'
            }
        }
            \int_{S}
                \abs{F_j f\br{\alpha,\vbeta}}
                \int_{S'}
                    \abs{F_j f\br{\alpha',\vbeta'}}
            \cdot\abs{
                        \ang{
                            \varphi_{\alpha,\vbeta}^{j},
                            \varphi_{\alpha',\vbeta'}^{j}
                        }
                    }
                d\alpha' d\mu\br{\vbeta'}
            d\alpha d\mu\br{\vbeta}.
     \end{equation}
    We treat diagonal terms and off-diagonal terms differently. Starting with the diagonal terms \eqref{eq_L2_est_diag}, we rewrite the expression
    \begin{equation*}
        \sum_{
            \br{I,\vgamma,S}\in\S
        }
            \int_{S}
                \int_{S}
                    \abs{F_j f\br{\alpha,\vbeta}}
                    \abs{
                        \ang{
                            \varphi_{\alpha,\vbeta}^{j},
                            \varphi_{\alpha',\vbeta'}^{j}
                        }
                    }^{\frac{1}{2}}
    \end{equation*}
    \begin{equation*}
             \cdot
                    \abs{F_j f\br{\alpha',\vbeta'}}
                    \abs{
                        \ang{
                            \varphi_{\alpha,\vbeta}^{j},
                            \varphi_{\alpha',\vbeta'}^{j}
                        }
                    }^{\frac{1}{2}}
                d\alpha' d\mu\br{\vbeta'}
            d\alpha d\mu\br{\vbeta}.
    \end{equation*}
    Cauchy-Schwarz then gives
    \begin{equation*}
        \leq
        \sum_{
            \br{I,\vgamma,S}\in\S
        }
            \int_{S}
                \abs{F_j f\br{\alpha,\vbeta}}^2
                \int_{S}
                    \abs{
                        \ang{
                          \varphi_{\alpha,\vbeta}^{j},
                            \varphi_{\alpha',\vbeta'}^{j}
                        }
                    }
                d\alpha' d\mu\br{\vbeta'}
            d\alpha d\mu\br{\vbeta}.
    \end{equation*}
    Controlling \eqref{eq_L2_est_diag} is reduced to showing that for all $\br{\alpha,\vbeta}\in S$,
    \begin{equation}\label{eq_L2_est_diag_core}
        \int_{S}
            \abs{
                \ang{
                   \varphi_{\alpha,\vbeta}^{j},
                    \varphi_{\alpha',\vbeta'}^{j}
                }
            }
        d\alpha' d\mu\br{\vbeta'}\lesssim 1.
    \end{equation}
    For the off-diagonal terms \eqref{eq_L2_est_off_diag}, we observe that
    \begin{equation*}
        \abs{F_j f \br{\alpha,\vbeta}},
        \abs{F_j f \br{\alpha',\vbeta'}}
        \leq \lambda
        \lesssim 
        \abs{I}^{-\frac{1}{2}}
        \nrm{F_j f}_{L^2_{\nu}\br{S}}.
    \end{equation*}
    After substitution, we can dominate \eqref{eq_L2_est_off_diag} by
    \begin{equation*}
        \sum_{
            \br{I,\vgamma,S}\in\S
        }
            \nrm{F_j f}_{L^2_{\nu}\br{S}}^2
        \cdot
            \frac{1}{\abs{I}}
            \int_{S}
                \sum_{
                    \substack{
                        \br{I',\vgamma',S'}\in\S\\
                        \br{I,\vgamma,S}\neq
                        \br{I',\vgamma',S'}\\
                        \gamma_j\leq \gamma_j'
                    }
                }
                \int_{S'}
                    \abs{
                        \ang{
                            \varphi_{\alpha,\vbeta}^{j},
                            \varphi_{\alpha',\vbeta'}^{j}
                        }
                    }
                d\alpha' d\mu\br{\vbeta'}
            d\alpha d\mu\br{\vbeta}.
      \end{equation*}
    Controlling \eqref{eq_L2_est_off_diag} is reduced to showing for \(\br{I,\vgamma,S}\in\S\),
    \begin{equation}\label{eq_L2_est_off_diag_core}
        \frac{1}{\abs{I}}
            \int_{S}
                \sum_{
                    \substack{
                        \br{I',\vgamma',S'}\in\S\\
                        \br{I,\vgamma,S}\neq
                        \br{I',\vgamma',S'}\\
                        \gamma_j\leq \gamma_j'
                    }
                }
                \int_{S'}
                    \abs{
                        \ang{
                            \varphi_{\alpha,\vbeta}^{j},
                            \varphi_{\alpha',\vbeta'}^{j}
                        }
                    }
                d\alpha' d\mu\br{\vbeta'}
            d\alpha d\mu\br{\vbeta}\lesssim 1.
    \end{equation}
    To prove \eqref{eq_L2_est_diag_core} and control the inner-most integral of \eqref{eq_L2_est_off_diag_core},
    we first focus on a fixed \(\br{\alpha,\vbeta}\in S\) and discard all points \(\br{\alpha',\vbeta'}\) in \(S'\) (in \eqref{eq_L2_est_diag_core}, we take \(S'=S\)) that don't contribute to the integral.
   Therefore, we may assume
    \begin{equation*}
        \ang{
            \varphi_{\alpha,\vbeta}^{j},
            \varphi_{\alpha',\vbeta'}^{j}
        }\neq 0.
    \end{equation*}
    This implies the frequency supports of the functions overlap, and thus, 
    \begin{equation*}
        \abs{\beta_j-\beta'_j}\leq \frac{4\varepsilon}{10}\br{d_\Gamma\br{\vbeta}+d_\Gamma\br{\vbeta'}}.
    \end{equation*}
    Recall that \(S\subseteq I\times W_{\vgamma,\frac{1}{\abs{I}}}^{<j}\) and \(S'\subseteq I'\times W_{\vgamma',\frac{1}{\abs{I'}}}^{<j}\). Since \(\gamma_j\leq\gamma_j'\), Lemma \ref{lem_L2_est_lac_scale_bd} guarantees the following relations in both \eqref{eq_L2_est_diag_core} and \eqref{eq_L2_est_off_diag_core}
    \begin{equation}\label{eq_L2_est_scale_bd}\\
         d_\Gamma\br{\vbeta}\leq  c d_\Gamma\br{\vbeta'},
    \end{equation}
    \begin{equation}\label{eq_L2_est_scale_eq}
        \abs{\beta_j-\gamma_j'}
        \sim 
        d_\Gamma\br{\vbeta'}
        \sim 
        \abs{\vbeta'-\vgamma'}.
     \end{equation} 
    On the one hand, due to the relation \eqref{eq_L2_est_scale_bd}, the standard wave-packet estimate gives
    \begin{equation*}
        \abs{
            \ang{
                \varphi_{\alpha,\vbeta}^{j},
                \varphi_{\alpha',\vbeta'}^{j}
            }
        }
        \lesssim
        d_\Gamma\br{\vbeta}\br{1+d_\Gamma\br{\vbeta}\abs{\alpha-\alpha'}}^{-N}.
    \end{equation*}
    On the other hand, the relation \eqref{eq_L2_est_scale_eq} implies that
    \begin{equation*}
        \mu\br{P_V\br{S'}}\! \leq  \! \!\int_{\left\{\vbeta'\in W_{\vgamma',0}\middle\vert \abs{\beta_j-\gamma'_j}\sim \abs{\vbeta'-\vgamma'}\right\}}
        \frac{d\mathcal{H}^2\br{\vbeta'}}{d_\Gamma\br{\vbeta'}^2}
        \lesssim \!
        \fint_{\left\{\vbeta'\in V\middle\vert \abs{\beta_j-\gamma'_j}\sim \abs{\vbeta'-\vgamma'}\right\}}
        d\mathcal{H}^2\br{\vbeta'}=1.
     \end{equation*}
    In combination, we obtain
    \begin{equation*}
        \int_{S'}
            \abs{
                \ang{
                   \varphi_{\alpha,\vbeta}^{j},
                    \varphi_{\alpha',\vbeta'}^{j}
                }
            }
        d\alpha' d\mu\br{\vbeta'}
        \!\lesssim  \!
        \int_{P_V\br{S'}}
            \int_{P_\R\br{S'}}
            d_\Gamma\br{\vbeta}\br{1+d_\Gamma\br{\vbeta}\abs{\alpha-\alpha'}}^{-N}
        d\alpha' d\mu\br{\vbeta'}
    \end{equation*}
    \begin{equation*}
        \lesssim 
        \int_{P_\R\br{S'}}
            d_\Gamma\br{\vbeta}\br{1+d_\Gamma\br{\vbeta}\abs{\alpha-\alpha'}}^{-N}
        d\alpha'\leq \int_\R \br{1+\abs{x}}^{-N}dx\sim 1
     \end{equation*}
    In particular, this proves \eqref{eq_L2_est_diag_core} by taking \(S'=S\). To address the issue with summation involved in \eqref{eq_L2_est_off_diag_core}, we utilize the relation \eqref{eq_L2_est_scale_bd} and infer that
    all points in \(\br{\alpha'',\vbeta''}\in S''\)
    from any other \(\br{I'',\vgamma'',S''}\) that contribute to the integral in \eqref{eq_L2_est_off_diag_core} satisfy also the following relation
    \begin{equation*}
        \abs{\beta_j'-\beta_j''}
        \leq 
        \abs{\beta_j-\beta_j'}+
        \abs{\beta_j-\beta_j''}
        \leq 
        \frac{4\varepsilon}{10}\br{ 2d_\Gamma\br{\vbeta} + d_\Gamma\br{\vbeta'}+d_\Gamma\br{\vbeta''}}
    \end{equation*}
    \begin{equation*}
        \leq  \frac{4\varepsilon}{10} \cdot\br{1+c}\br{d_\Gamma\br{\vbeta'}+d_\Gamma\br{\vbeta''}}<
        \rho \br{d_\Gamma\br{\vbeta'}+d_\Gamma\br{\vbeta''}}.
     \end{equation*}
    This forces that all pairs of points \(\br{\alpha',\vbeta'}\in S'\) and \(\br{\alpha'',\vbeta''}\in S''\) with distinct \(\br{I',\vgamma',S'}\) and \(\br{I'',\vgamma'',S''}\) in \eqref{eq_L2_est_off_diag_core} have their physical components separated in the following way
    \begin{equation*}
        \abs{\alpha-\alpha'}, \abs{\alpha-\alpha''}\geq 2\abs{I},\quad
        \abs{\alpha'-\alpha''}\geq 2 \br{\abs{I'}\wedge \abs{I''}}>0.
    \end{equation*}
    In other words, the sets \(3I\), \(P_\R\br{S'}\), and
    \(P_\R\br{S''}\) are disjoint from one another.
    Finally, we combine the result above,
    \begin{align*}
        &
        \frac{1}{\abs{I}}
        \int_S 
            \sum_{
                \substack{
                    \br{I',\vgamma',S'}\in\S\\
                    \br{I,\vgamma,S}\neq\br{I',\vgamma',S'}\\
                    \gamma_j\leq \gamma_j'
                }
            }
                \int_{S'}
                    \abs{
                        \ang{
                            \varphi_{\alpha,\vbeta}^{j},
                          \varphi_{\alpha',\vbeta'}^{j}
                        }
                    }
                d\alpha' d\mu\br{\vbeta'}
        d\alpha d\mu\br{\vbeta}\\
        \lesssim &
        \frac{1}{\abs{I}}
        \int_S 
            \sum_{
                \substack{
                    \br{I',\vgamma',S'}\in\S\\
                    \br{I,\vgamma,S}\neq\br{I',\vgamma',S'}\\
                    \gamma_j\leq \gamma_j'
                }
            }
            \int_{\pi_\R\br{S'}}
                d_\Gamma\br{\vbeta}\br{1+d_\Gamma\br{\vbeta}\abs{\alpha-\alpha'}}^{-N}
            d\alpha'
        d\alpha d\mu\br{\vbeta}\\
        \leq &
        \fint_I
            \int_{W_{\gamma,\frac{1}{\abs{I}}}}
                \int_{3I^c}
                    d_\Gamma\br{\vbeta}
                    \br{1+d_\Gamma\br{\vbeta}\abs{\alpha-\alpha'}}^{-N}
                d\alpha'
            d\mu\br{\vbeta}
        d\alpha\\
        \lesssim &
        \int_{W_{\gamma,\frac{1}{\abs{I}}}}
            \br{1+d_\Gamma\br{\vbeta}\abs{I}}^{2-N}
            \fint_I
                \int_{3I^c}
                    d_\Gamma\br{\vbeta}
                    \br{1+d_\Gamma\br{\vbeta}\abs{\alpha-\alpha'}}^{-2}
                d\alpha'
            d\alpha
        d\mu\br{\vbeta}\\
        \lesssim &
        \int_{W_{\gamma,\frac{1}{\abs{I}}}}
            \br{d_\Gamma\br{\vbeta}\abs{I}}^{2-N}
        d\mu\br{\vbeta}\\
        \lesssim &
        \int_{B\br{\vgamma,\frac{1}{\abs{I}}}^c}
            \abs{I}^{-N}\abs{\vbeta-\vgamma}^{-N}
            \abs{I}^2
        d\mathcal{H}^2\br{\vbeta}
        =\int_{B\br{0,1}^c}
            \abs{\vbeta}^{-N}
        d\mathcal{H}^2\br{\vbeta}\sim 1.
    \end{align*}
\end{proof}

We close this section with the proof of Proposition \ref{prop_L2_emb}.
\begin{equation}
    \Omega \cap |F_{j}f|^{-1}(\lambda ,\infty)=\bigsqcup_{k\in \mathbb{N}}\Omega \cap |F_{j}f|^{-1}(2^{k-1}\lambda ,2^{k}\lambda]
\end{equation}
For each $k$, apply Proposition \ref{prop_select} to the set \(\Omega \cap |F_{j}f|^{-1}(2^{k-1}\lambda,2^{k}\lambda]\) with the threshold $2^{k-1}\lambda$, we obtain a set of countable points $\P_{k}$ and a countable collection of tents $\T_{k}$ satisfying the properties stated in Proposition \ref{prop_select}. Define
$\P_{\infty}=\underset{k\in \mathbb{N}}{\bigcup}\P_{k}$ and $\T_{\infty}=\underset{k\in \mathbb{N}}{\bigcup}\T_{k}$. By \eqref{select1-2},
\begin{equation}\label{eq875}
    \sum_{(I,\vgamma)\in \T_{\infty}}|I|\sim \sum_{(\alpha,\vbeta)\in \P_{\infty}}d_{\Gamma}(\vbeta)^{-1}\lesssim \sum_{k\in \mathbb{N}}\sum_{(\alpha,\vbeta)\in \P_{k}}\frac{|F_{j}f(\alpha,\vbeta)|^{2}}{d_{\Gamma}(\vbeta)(2^{k-1}\lambda)^{2}}.
\end{equation}
By Lemma \ref{prop_Linfty_comp_est}, we can further bound \eqref{eq875} by
\begin{equation}
    \sum_{k\in \mathbb{N}}2^{-2(k-1)}\lambda^{-2}\sum_{(\alpha,\vbeta)\in \P_{k}}\frac{|F_{j}f(\alpha,\vbeta)|^{2}}{d_{\Gamma}(\vbeta)}\lesssim \frac{\|f\|_{L^{2}}^{2}}{\lambda^{2}}
\end{equation}
As a direct consequence of \eqref{select1-1}, we have
\begin{equation}\label{eq_Linfty_comp_threshold_less_lambda}
    \nrm{
        1_{\Omega\setminus \bigcup_{T\in\T_\infty}D_T} F_j f
    }_{L^\infty}\leq \lambda.
\end{equation}
Next, define $\widetilde{\Omega}=\Omega \setminus \underset{T\in \T_{\infty}}{\bigcup}D_{T} \subseteq |F_{j}f|^{-1}[0,\lambda]$. Applying Proposition \ref{prop_selectL2} to $\widetilde{\Omega}$, we obtain a countable collection $\mathbf{T}_{left}$ of tents and a countable collection $\mathbf{S}_{left}$ of the form \(\br{I,\vgamma, S}\) with $(I,\vgamma)$ a tent and $S$ a measurable subset of $ \Omega\cap I\times \left(W_{\vgamma,\frac{1}{\abs{I}}}\setminus U_{\vgamma}^{j}\right)^{<j}$ satisfying Properties stated in Proposition \ref{prop_selectL2}. By Lemma \ref{prop_L2_comp_est},
\begin{equation}\label{eq877}
    \sum_{(I,\vgamma)\in \T_{left}}|I|\sim \sum_{(I,\vgamma,S)\in \S_{left}}|I|\lesssim \sum_{(I,\vgamma,S)\in \S_{left}}\frac{\|F_{j}f\|_{L^{2}_{\nu}(S)}^{2}}{\lambda^{2}}\lesssim \frac{\|f\|_{L^{2}}^{2}}{\lambda^{2}}.
\end{equation}
Applying the dual version of Proposition \ref{prop_selectL2} to $\left(\Omega \setminus \underset{T\in \T_{\infty}}{\bigcup}D_{T}\right) \setminus \underset{T\in \T_{left}}{\bigcup}D_{T} $ with \(\left(W_{\vgamma,\frac{1}{\abs{I}}}\setminus U_{\vgamma}^{j}\right)^{<j}\) replaced by \(\left(W_{\vgamma,\frac{1}{\abs{I}}}\setminus U_{\vgamma}^{j}\right)^{>j}\), we obtain another countable collection $\mathbf{T}_{right}$ of tents and a countable collection $\mathbf{S}_{right}$ of triples which satisfy equation similar to \eqref{eq877}. 
Define $\T=\T_{\infty}\cup \T_{left} \cup \T_{right}$. 
Direct calculation gives \eqref{eq_Tlarge_est}
\begin{equation}
    \sum_{\br{I,\vgamma}\in\T}\abs{I}=
    \sum_{\br{I,\vgamma}\in\T_\infty}\abs{I}+
    \sum_{\br{I,\vgamma}\in\T_{left}}\abs{I}+
    \sum_{\br{I,\vgamma}\in\T_{right}}\abs{I}
    \lesssim
    \frac{\nrm{f}_{L^2}^2}{\lambda^2}.
\end{equation}
% which completes the proof of \eqref{eq_Tlarge_est}.
Finally, notice that by Proposition \ref{prop_selectL2}, \(\T\) satisfies \eqref{select2-1}
\begin{equation}\label{eq_left_small}
    |I|^{-\frac{1}{2}}\nrm{
        1_{\Omega\setminus\bigcup_{T\in\T}D_T} F_jf
    }_{L^2_{\nu}\br{ I\times (W_{\vgamma,1/|I|}\setminus U_{\vgamma}^{j})^{<j}}}\leq \frac{\lambda}{\sqrt{2}},
\end{equation}
and the dual statement of \eqref{select2-1}
\begin{equation}\label{eq_right_small}
    |I|^{-\frac{1}{2}}\nrm{
        1_{\Omega\setminus\bigcup_{T\in\T}D_T} F_jf
    }_{L^2_{\nu}\br{ I\times (W_{\vgamma,1/|I|}\setminus U_{\vgamma}^{j})^{>j}}}\leq \frac{\lambda}{\sqrt{2}}.
\end{equation}
for all \(\br{I,\vgamma}\).
Combining \eqref{eq_Linfty_comp_threshold_less_lambda}, \eqref{eq_left_small}, and \eqref{eq_right_small} verifies that
\begin{equation}
    \nrm{1_{\Omega\setminus\bigcup_{T\in\T}D_T}F_j f}_{S^j}\leq  \lambda \vee \sqrt{\br{\frac{\lambda}{\sqrt{2}}}^2+\br{\frac{\lambda}{\sqrt{2}}}^2}=\lambda.
\end{equation}
That is, \(\T\) also satisfies \eqref{eq_Omegasmall_con}.
% Then $\T$ is a desired collection which satisfies \eqref{eq_Tlarge_est}, \eqref{eq_Omegasmall_con} and thus completes the proof of Proposition \ref{prop_L2_emb}.
This completes the proof of Proposition \ref{prop_L2_emb}.

\section{Proof of Proposition \ref{boundmodelform}: Bound of Model Form}\label{sec_proofmodelbound}
Notice that there is a decomposition \(V\setminus \Gamma=\bigsqcup_{\iota=1}^N V_\iota\) with \(V_\iota\) defined iteratively
\begin{equation}
    V_0:=\Gamma,\quad V_\iota:=\left\{ 
        \vbeta\in V\setminus \bigcup_{j<\iota}V_j 
    \: : \:
        d_\Gamma\br{\vbeta}=d_{\Gamma_\iota}\br{\vbeta}
    \right\}.
\end{equation}
As a result, we obtain a decomposition on the trilinear form \(\Lambda_m=\sum_{\iota=1}^N\Lambda_{m,\iota}\), where
\begin{equation}\label{eq_lambda_eq_modelform}
\Lambda_{m,\iota}\br{f_1,f_2,f_3}:=\int_{V_{\iota}}\int_{\mathbb{R}^{3}}K(\valpha, \vbeta)\cdot \prod_{j=1}^{3}\left(F_j f_{j}\right)(\alpha_{j},\vbeta) d\valpha d\mu (\vbeta).
\end{equation}
Thus, by triangle inequality, it suffices to deal with a single piece \(\Lambda_{m,\iota}\). Moreover, via a standard limiting argument, it suffices to perform the estimate with \(\R^3\) replaced by a set of the form \(\Br{-A,A}^3\) and \(V\) replaced by a compact subset \(V'\subseteq V\setminus \Gamma\). For simplicity, we may fix \(\iota\) and assume \(\Gamma=\Gamma_\iota\), \(K=1_{\Br{-A,A}^3\times \br{V_\iota\cap V'}}K\), and \(\Lambda_m=\Lambda_{m,\iota}\).

The remaining part is parallel to the restricted weak type estimate in \textit{(4.2)} of \cite{kovavc2015dyadic}.
Let $E_{j}\subseteq \mathbb{R}$ be measurable sets and $|f_{j}|\leq 1_{E_{j}}$, it suffices to show 
\begin{equation}\label{resweakest}
|\Lambda_{m}(f_{1},f_{2},f_{3})|\lesssim a_{2}^{\frac{1}{2}}a_{3}^{\frac{1}{2}}(1+\log \frac{a_{1}}{a_{2}}),
\end{equation}
where $a_{j}=|E_{\sigma (j)}|$ is a decreasing rearrangement $(a_{1}>a_{2}>a_{3})$. Since for any $\varepsilon>0$ and \(x\geq 1\), we have the the asymptotic
$
1+\log x =o(x^{\varepsilon}),
$
\begin{equation*}
a_{2}^{\frac{1}{2}}a_{3}^{\frac{1}{2}}(1+\log \frac{a_{1}}{a_{2}})<Ca_{1}^{\varepsilon}a_{2}^{\frac{1}{2}-\varepsilon}a_{3}^{\frac{1}{2}}.
\end{equation*}
Then, by interpolating the restricted weak type estimates, we obtain the strong bound in the local $L^{2}$ range. By $L^{2}$ normalization, $\widetilde{f_{j}}:=\frac{f_{j}}{|E_{j}|^{\frac{1}{2}}}$, \eqref{resweakest} reduces to
\begin{equation*}
|\Lambda_{m}(\widetilde{f}_{1},\widetilde{f}_{2},\widetilde{f}_{3})|\lesssim a_{1}^{\-\frac{1}{2}}(1+\log \frac{a_{1}}{a_{2}}).
\end{equation*}
Note that by Proposition \ref{prop_Linfty_emb} (Global estimate), 
\begin{equation*}
\|F_{j}\widetilde{f}_{j}\|_{S^{j}}\lesssim \|\widetilde{f}_{j}\|_{L^{\infty}}\leq \|E_{j}\|^{-\frac{1}{2}}=a_{\sigma^{-1}(j)}^{-\frac{1}{2}},
\end{equation*}
and by normalization, $\|\widetilde{f}_{j}\|_{L^{2}}\leq 1$.
  Let $n_{j}$ be the integer such that $2^{n_{j}-1}<a_{j}^{-\frac{1}{2}}\leq 2^{n_{j}}$. By design, $n_{1}\leq n_{2}\leq n_{3}$. We perform Proposition \ref{prop_L2_emb} (Bessel type estimate) iteratively.
  Let \(\widetilde{\Omega}_{n_3}=\Br{-A,A}^3 \times V' \) and \(\pi_i:\mathbb{R}^{3}\times V\to \mathbb{R}\times V\) be the following projection:
    \begin{equation*}
\pi_{i}\br{\valpha,\vbeta}:=\br{\alpha_i,\vbeta}.
    \end{equation*}
  Given a compact set \(\widetilde{\Omega}_{n}\subset \Br{-A,A}^3\times V'\)
  with the properties that for \(i,i'\in\BR{1,2,3}\),
  \begin{equation}\label{eq_omeg_omeh_omeg}
      \Omega_n:=\pi_i\widetilde{\Omega}_n=\pi_{i'}\widetilde{\Omega}_n
  \end{equation}
  and for all \(j\in\BR{1,2,3}\),
  \begin{equation}\label{eq_omeg_small_2n}
      \nrm{1_{\Omega_n}F_j\widetilde{f}_j}_{S^j}\lesssim 2^n,
  \end{equation}
  we apply Proposition \ref{prop_L2_emb} with \(\lambda=2^{n-1}\) for each $j\in \{1,2,3\}$ and obtained a collection of tents $\mathbf{T}_{n,j}$ such that
  \begin{equation*}
  \sum_{\br{I,\vgamma}\in \mathbf{T}_{n,j}}|I|\lesssim 2^{-2n}.
    \end{equation*}
    Let $\T_{n}:=\bigcup_{j=1}^{3}\T_{n,j}$. By triangle inequality,
    \begin{equation*}
        \sum_{\br{I,\vgamma}\in \mathbf{T}_n}|I|\lesssim 2^{-2n}.
    \end{equation*}
 %  and for $T=(I,\vgamma)\in \textbf{T}_{n}$,
 %  \begin{equation*}
 %  \|1_{\Omega_n}F_{j}\widetilde{f}_{j}\|_{S^{j}(T)}\lesssim \min \br{2^{n},2^{n_{\sigma^{-1}(j)}}}.
 % \end{equation*}
  We can associate the data $T=(I,\vgamma)$ with a region $D^i_T$ defined by
  \begin{equation*}
    D_T^i:=(I\ve_{i}\oplus \ve_{i}^{\perp})\times W_{\vgamma,1/|I|}
  \end{equation*}
  and define the following compact subset:
  \begin{equation*}
      \widetilde{\Omega}_{n-1}:=\widetilde{\Omega}_{n}\setminus  \bigcup_{T\in\T_n}\bigcup_{i=1}^3 D^i_T.
  \end{equation*}
  By construction, the set \(\widetilde{\Omega}_{n-1}\) satisfies \eqref{eq_omeg_omeh_omeg} and \eqref{eq_omeg_small_2n} with \(n\) replaced by \(n-1\). Through this iteration process, we obtain a nested sequence of compact sets
  \begin{equation*}
      \widetilde{\Omega}_{n_3}\supseteq \cdots \supseteq \widetilde{\Omega}_{n} \supseteq \widetilde{\Omega}_{n-1}\supseteq \cdots
  \end{equation*}
  and a countable collection of tents \(\T:=\bigcup_{n\leq n_3}\T_n\).
  Using the identity \eqref{eq_lambda_eq_modelform}, we dominate \(\Lambda_m\br{\widetilde{f}_1,\widetilde{f}_2,\widetilde{f}_3}\) in the following manner
    \begin{equation}
        \abs{
            \Lambda_m\br{\widetilde{f}_1,\widetilde{f}_2,\widetilde{f}_3}
        }
        \leq
        \sum_{n\leq n_3}\sum_{T\in\T_n}\sum_{i=1}^3
        \left\|1_{\widetilde{\Omega}_{n}}K(\valpha, \vbeta)\cdot \prod_{j=1}^{3}(F_{j}f_{j})(\alpha_{j},\vbeta)\right\|_{L^1_{\valpha, \mu(\vbeta)}(D^i_T)}.
    \end{equation}
  Apply Proposition \ref{propTentest} (Tent estimate), we obtain
\begin{equation*}
|\Lambda_{m}(\widetilde{f}_{1},\widetilde{f}_{2},\widetilde{f}_{3})| \lesssim \sum_{n\leq n_{3}}\sum_{T\in \textbf{T}_{n}}|I_{T}|\prod_{i=1}^{3} \| 1_{\Omega_n}F_{j}\widetilde{f}_{j}\|_{S^{j}(T)}
\lesssim \sum_{n\leq n_{3}}2^{-2n}\prod_{i=1}^{3}\min (2^{n},2^{n_{i}})
\end{equation*}
\begin{equation*}
=\sum_{n\leq n_{1}}2^{n}+\sum_{n_{1}\leq n< n_{2}}2^{n_{1}}+\sum_{n_{2}\leq n<n_{3}}2^{n_{1}}\cdot 2^{n_{2}-n}
\lesssim 2^{n_{1}}+(n_{2}-n_{1})2^{n_{1}}+2^{n_{1}}
\end{equation*}
\begin{equation*}
    =2^{n_{1}}(2+n_{2}-n_{1})
\lesssim a_{1}^{-\frac{1}{2}}(2+\log \frac{a_{1}}{a_{2}}).
\end{equation*}
which completes the proof of Proposition \ref{boundmodelform}.

%\section{Proof of Corollary \ref{maincor}}
%\input{proof of maincor}

\newpage
%\printbibliography[title={References}]

\bibliographystyle{plain}
\bibliography{ref}

\begin{align*}
    &\textsc{ Jiao Chen, Chongqing Normal University, China}\\
    &\textit{Email address:}\: \:\textbf{chenjiaobnu@163.com}\\
    \\
&\textsc{ Martin Hsu, Purdue University, USA}\\
 &\textit{Email address:}\: \:\textbf{hsu263@purdue.edu}\\
 \\
&\textsc{ Fred Yu-Hsiang Lin, University of Bonn, Germany}\\
 &\textit{Email address:}\: \:\textbf{ fredlin@math.uni-bonn.de}\\
\end{align*}

\end{document}